\documentclass[12pt]{article}

\usepackage[margin=1in]{geometry}

\usepackage{amsmath,amsthm} 
\usepackage{amssymb}
\usepackage[all]{xypic} 
\usepackage{amsbsy}
\usepackage{graphicx}  
\usepackage{subfigure} 
\usepackage[all,cmtip]{xy} 
\usepackage{multirow}  
\usepackage{multicol}
\usepackage{rotating} 
\usepackage{comment} 
\usepackage{tikz, tikz-cd} 
\usetikzlibrary{shapes.geometric, arrows, arrows.meta, decorations.markings, decorations.pathmorphing, shapes} 
\usepackage{relsize} 
\usepackage{hyperref} 
\usepackage{xcolor} 
\usepackage{ stmaryrd } 
 
\numberwithin{equation}{section}

\newtheorem{thm}{Theorem}[section]
\newtheorem{cor}[thm]{Corollary}

\newtheorem{lem}[thm]{Lemma}

\theoremstyle{definition} 
\newtheorem{defn}[thm]{Definition}
\newtheorem{rem}[thm]{Remark}
\newtheorem{exam}[thm]{Example}

\newtheorem{exercise*}{Solution to Exercise}

\newcommand{\travis}[1]{{\color{purple} \textbf{T:} [#1]}}

\newcommand{\bA}{\mathbb{A}} 
\newcommand{\bC}{\mathbb{C}}

\newcommand{\bP}{\mathbb{P}}
 
\newcommand{\bR}{\mathbb{R}}

\newcommand{\bZ}{\mathbb{Z}}
\newcommand{\bN}{\mathbb{N}}

\newcommand{\cM}{\mathcal{M}}

\newcommand{\Hom}{\operatorname{Hom}}

\newcommand{\Rep}{\operatorname{Rep}}
\newcommand{\Mat}{\operatorname{Mat}}

\newcommand{\Spec}{\operatorname{Spec}}
\newcommand{\Proj}{\operatorname{Proj}}

\newcommand{\Pic}{\operatorname{Pic}}
\newcommand{\GL}{\operatorname{GL}}

\newcommand{\sm}{\operatorname{sm}}
\newcommand{\sgn}{\operatorname{sign}}

\begin{document}

\title{Nonprojective 
crepant resolutions of 
quiver varieties}
\author{Daniel Kaplan and Travis Schedler}
\date{May 2025}
\maketitle

\begin{abstract}
In this paper, we construct a large class of examples of proper, nonprojective crepant resolutions of singularities for Nakajima quiver varieties. These include four and six dimensional examples and examples with $Q$ containing only three vertices. There are two main techniques: by taking a locally projective resolution of a projective partial resolution as in our previous work \cite{KS24}, and more generally by taking quotients of open subsets of representation space which are not stable loci, related to Arzhantsev--Derental--Hausen--Laface's construction in the setting of Cox rings \cite{ADHL}. By the latter method we exhibit a proper crepant resolution that does not factor through a projective partial resolution. Most of our quiver settings involve one-dimensional vector spaces, hence the resolutions are toric hyperk\"ahler, which were studied from a different point of view in  Arbo--Proudfoot \cite{AP16}.  

This builds on the classification of  \emph{projective} crepant resolutions of a large class of quiver varieties in \cite{BCS23} and the classification of \emph{proper} crepant resolutions for the hyperpolygon quiver varieties in \cite{Hubbard}. 
\end{abstract}

\section{Introduction}

\subsection{Crepant resolutions of singularities} 

\underline{Background}: A \emph{resolution of singularities} is a proper, birational morphism $\pi: \tilde X \to X$ with smooth source.  Suppose that $X$ 
is normal and has a symplectic form $\omega$ on its smooth locus, $X^{\sm} \subset X$. Then $X$ is a \emph{symplectic singularity} if the pullback $\pi^*\omega$ extends to a regular (but possibly degenerate) two-form on $\tilde X$, see \cite[Definition 1.1]{Beauville}. Further, the resolution $\pi$ is called \emph{symplectic} if
the resulting two-form on $\tilde X$ is nondegenerate (hence a symplectic form). 
 
Examples of symplectic resolutions include the Springer resolution of the nilpotent cone of a semisimple Lie algebra, the Hilbert--Chow morphism from the Hilbert scheme of points to its underlying symmetric product, and the minimal resolution of a du Val singularity. Symplectic resolutions have been studied in various contexts; see surveys such as  \cite{Fu}, \cite{Kaledin-survey}, \cite{Kamnitzer-survey}, and \cite{Okounkov-survey}.

Many of these examples can be recast using the unifying combinatorial framework of Nakajima quiver varieties. Bellamy--Schedler proved that all Nakajima quiver varieties are symplectic singularities and gave a complete classification of which admit a symplectic resolution of singularities \cite[Theorem 1.2, Theorem 1.5]{BS21}.
Bellamy--Craw--Schedler \cite[Corollary 1.3]{BCS23} proved that, under mild conditions, all such \emph{projective} symplectic resolutions are given by variation of stability parameter $\pi: \cM_{\theta'}(Q, \alpha) \to \cM_{\theta}(Q, \alpha)$, where the resolution only depends on the chamber in which $\theta'$ lies of a hyperplane arrangement complement, up to the action of a finite group. \\

\noindent \underline{Objective}: The goal of this paper is to construct a wide class of explicit examples of nonprojective (proper) symplectic resolutions of Nakajima quiver varieties. Most of our examples involve all dimensions at vertices equal to one, so that the resulting quiver variety is toric hyperk\"{a}hler\footnote{Although the term \emph{hypertoric} is commonly used in the literature, we adopt the earlier term \emph{toric hyperk\"{a}hler}.}. We give evidence
that  \emph{most} proper symplectic resolutions of Nakajima quiver varieties (for large enough quivers) are not projective.


\begin{rem}(Symplectic vs crepant)
These examples fall into the more general problem of classifying (proper, but not necessarily projective) crepant resolutions of singularities.  A \emph{crepant} resolution of a Gorenstein singularity is one which pulls the canonical line bundle of the singularity back to the canonical line bundle of the resolution.
Note that every symplectic singularity is Gorenstein \cite[Proposition 1.3]{Beauville}. Since the symplectic structure induces a canonical volume, it follows that  every symplectic resolution is crepant. Conversely, every crepant resolution of a symplectic singularity  is symplectic (for a unique symplectic structure on the source), by \cite[Proposition 3.2]{Kaledin_Crepant}. 
In this paper, although we only consider Nakajima quiver varieties, which are symplectic singularities by \cite{BS21}, we prefer to use the terminology \emph{crepant resolution} in most of the paper.
\end{rem}

\noindent \underline{History}: Arbo--Proudfoot produce the first known examples of proper, nonprojective crepant resolutions of toric hyperk\"{a}hler varieties: see \cite[Corollary 6.8 and Examples 6.9, 6.10]{AP16}. Hubbard gave a complete classification of all proper crepant resolutions of hyperpolygon cones, which are Nakajima quiver varieties for an unframed star-shaped quiver with central vertex having dimension two and outer vertices dimension one \cite[Theorem 1.2]{Hubbard}. 

 In \cite[Corollary 1.8]{KS24}, the authors produce locally-projective resolutions of not-necessarily conical varieties with a nice stratification. One can use this procedure to construct nonprojective crepant resolutions, as we demonstrate below. \\

 \noindent \underline{Motivation}: 
From the viewpoint of algebraic geometry, it makes sense to consider proper resolutions, not merely projective ones. Concretely, these arise when one wishes to glue together local projective resolutions as in \cite{KS24}:  
 there are obstructions to the result being projective. 
 For projective symplectic resolutions, elements of their classification are understood: for instance, thanks to \cite{BCHM}, any symplectic singularity locally admits only finitely many projective crepant resolutions, and one can explicitly describe them via the movable fan, studied in more detail in \cite{Namikawa}. 
 
  Until recently it seemed quite difficult to understand the proper case; even basic questions, such as whether there are finitely many proper crepant resolutions up to isomorphism of a given singularity, or if different such resolutions are diffeomorphic or otherwise equivalent, seemed difficult (to us at least).  Aside from in the aforementioned results, few concrete nonprojective examples were produced.  By producing a large class of explicit concrete examples in the basic case of quiver varieties, it should be possible to gain an understanding of this more flexible, plentiful class of symplectic resolutions. In future work we plan to study their properties and classification in more detail.
  

 \begin{rem}(Partial resolutions)
We will also use the notion of a \emph{partial resolution} of singularities. This is a proper birational morphism, dropping the assumption that the source be smooth.  A partial resolution of a Gorenstein singularity is \emph{crepant} if  the canonical bundle pulls back to the canonical bundle.  A partial resolution of a symplectic singularity is  \emph{symplectic} if the pullback of the symplectic form over the smooth locus extends to a symplectic form on the smooth locus of the source\footnote{This extension is unique, and it makes the source a symplectic singularity.}. As before, a partial resolution of a symplectic singularity is crepant if and only if it is symplectic, so we usually will stick to the term \emph{crepant}. 
\end{rem}

 
    

\subsection{Constructions using locally projective resolutions} 
\begin{minipage}{0.7\textwidth}
Let $Q$ be the $D_4$ quiver with framing, shown to the right. Fix the dimension vector $\alpha = (1, 1, 1, 1, 1)$ and framing dimension $\beta = (0, 1, 1, 1, 1)$. Let $\theta = (2, -1, -1, -1, -1)$ and consider the Nakajima quiver variety $\cM_{\theta}(Q, \alpha)$, see Definition \ref{defn:Nakajima} and Subsection \ref{ss:unframed} below.

$\cM_{\theta}(Q, \alpha)$ has ${4 \choose 2} = 6$ singular points consisting of decomposable representations $V = W \oplus W'$ with 
\begin{align*}
\dim(W) \in \{ & (1, 1, 1, 0, 0), (1, 1, 0, 1, 0), (1, 1, 0, 0, 1), \\
&(1, 0, 1, 1, 0), (1, 0, 1, 0, 1), (1, 0, 0, 1, 1) \}. \\
\end{align*}
\end{minipage} \hspace*{0.6cm}
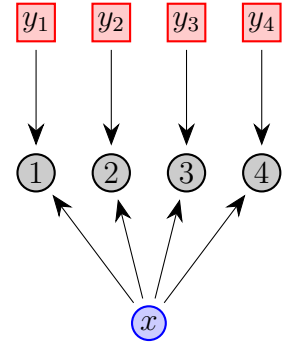
\begin{minipage}{0.15\textwidth}
\begin{tikzpicture}
    \node[circle, fill=blue!20, draw = blue,  thick, circle, inner sep = 0.3ex] (central) at (5, 0) {$x$};
    \node[circle, fill=black!20, draw = black, thick, circle, inner sep = 0.3ex] (v1) at (5-3/2, 2) {1};
    \node[circle, fill=black!20, draw = black, thick, circle, inner sep = 0.3ex] (v2) at (5-1/2, 2) {2};
    \node[circle, fill=black!20, draw = black, thick, circle, inner sep = 0.3ex] (v3) at (5+1/2, 2) {3};
    \node[circle, fill=black!20, draw = black,  thick, circle, inner sep = 0.3ex] (v4) at (5+3/2, 2) {4};
    \node[rectangle, minimum size=.5cm, fill=red!20, draw = red,  thick, inner sep = 0.3ex] (y1) at (5-3/2, 4) {$y_1$\vspace*{1cm}};
     \node[rectangle, minimum size=.5cm, fill=red!20, draw = red,  thick, inner sep = 0.3ex] (y2) at (5-1/2, 4) {$y_2$};
      \node[rectangle, minimum size=.5cm, fill=red!20, draw = red,  thick, inner sep = 0.3ex] (y3) at (5+1/2, 4) {$y_3$};
       \node[rectangle, minimum size=.5cm, fill=red!20, draw = red,  thick, inner sep = 0.3ex] (y4) at (5+3/2, 4) {$y_4$};
    \draw [-{Stealth[length=3mm]}, shorten >=1mm, shorten <=1mm]  (central) -- (v1);
    \draw [-{Stealth[length=3mm]}, shorten >=1mm, shorten <=1mm] (central) -- (v2);
    \draw [-{Stealth[length=3mm]}, shorten >=1mm, shorten <=1mm]  (central) -- (v3);
    \draw [-{Stealth[length=3mm]}, shorten >=1mm, shorten <=1mm]  (central) -- (v4);
    \draw [-{Stealth[length=3mm]}, shorten >=1mm, shorten <=1mm]  (y1) -- (v1);
    \draw [-{Stealth[length=3mm]}, shorten >=1mm, shorten <=1mm]  (y2) -- (v2);
    \draw [-{Stealth[length=3mm]}, shorten >=1mm, shorten <=1mm]  (y3) -- (v3);
    \draw [-{Stealth[length=3mm]}, shorten >=1mm, shorten <=1mm]  (y4) -- (v4);
\end{tikzpicture}
\end{minipage}

\noindent Each singular point is locally analytically isomorphic to the zero matrix in the cone: 
\[
\mathcal{N} := \{ M \in \Mat_{4 \times 4}(\bC) : M^2 = 0, \ \text{rank}(M) \leq 1 \}.
\]
Therefore, each singular point has a neighborhood that admits two crepant resolutions: the parabolic Springer resolution $T^*(\bP^3) \to \mathcal{N}$ and its dual $T^*(\bP^3)^{\vee} \to \mathcal{N}$.
All $2^6 = 64$ choices of local resolutions glue to locally projective (and hence proper) resolutions, see \cite[Cor 3.39]{KS24}. 

We prove in Theorem \ref{thm:four-point} and Remark \ref{r:18-resolutions} below that 18 of these resolutions do not come from variation of geometric invariant theory (GIT) quotient, while the other 46 do. By \cite{BCS23} all \emph{projective} crepant resolutions of $\mathcal{M}_{\theta}(Q, \alpha)$ are given by  variation of GIT, so these 18 crepant resolutions are not projective.

We give many generalizations of this example. Notably:
\begin{itemize}
    \item We show that one can enlarge the vertex set and, when it is large enough, most proper crepant resolutions are not projective (Corollary \ref{cor:probability});
    \item We give an example from \cite{AP16} which is four-dimensional (Section \ref{ss:Proudfoot}), using eight vertices, with isolated singularities looking like nilpotent $3 \times 3$ matrices of rank at most one;
    \item We give examples of quivers with just three and four vertices with dimension vector greater than one (Section \ref{ss:fewer-vertices}).
\end{itemize}

\subsection{Constructions using good quotients of open subsets}
    It is convenient to construct the locally projective resolutions of $\cM_\theta(Q,\alpha)$ via good quotients of the form $U/\!/G_\alpha$ for $U \subseteq \mu^{-1}(0)$ an open subset, which need not be the semistable locus for a stability condition. In this sense the quotient we obtain generalizes the GIT construction. An analogous construction was considered in \cite{ADHL}, building Mori dream spaces as torus quotients of the  spectrum of their Cox ring.  Here, analogously to the construction of Mori dream spaces in \cite{BCS23}, we instead take quotients by non-abelian groups,  but  unlike in \cite{BCS23}, we do not restrict to projective morphisms given by the GIT construction. The result, as we demonstrate, is a flexible technique for constructing a large class of proper crepant resolutions. 

    We study in particular an example where $Q$ is a five-pointed star with framings of one at each external vertex (shown at right). In this case, $\cM_{0}(Q, \alpha)$ has a proper resolution $X \to \cM_{0}(Q, \alpha)$ that does not factor through the partial resolution $\cM_{\theta}(Q, \alpha) \to \cM_0(Q, \alpha)$ for any $\theta$, see Theorem \ref{t:five-point-git}. 
    
\noindent
\begin{minipage}{0.67\textwidth}
\hspace*{1em} Consequently, $X$ cannot be obtained by iteratively taking locally projective crepant partial resolutions as in the 4-pointed star example. This is because every locally projective resolution of a cone with a finite stratification as in \cite{KS24} is isomorphic to a projective one, given by straightening out the resolution near the cone point  (this includes every Nakajima quiver variety, and more generally every conical Poisson variety with finitely many symplectic leaves, see \emph{op.~cit.}). Since every projective resolution of $\cM_0(Q,\alpha)$ is of the form $\cM_\theta(Q,\alpha)$ thanks to \cite{BCS23}, we deduce the assertion.
\end{minipage} \hspace{.4cm}
\begin{minipage}{0.15\textwidth}
\begin{tikzpicture}
    \node[circle, fill=blue!20, draw = blue,  thick, circle, inner sep = 0.3ex] (central) at (5, 0) {$x$};
    \node[circle, fill=black!20, draw = black, thick, circle, inner sep = 0.3ex] (v1) at (5-2, 2) {1};
    \node[circle, fill=black!20, draw = black, thick, circle, inner sep = 0.3ex] (v2) at (5-1, 2) {2};
    \node[circle, fill=black!20, draw = black, thick, circle, inner sep = 0.3ex] (v3) at (5+0, 2) {3};
    \node[circle, fill=black!20, draw = black,  thick, circle, inner sep = 0.3ex] (v4) at (5+1, 2) {4};
        \node[circle, fill=black!20, draw = black,  thick, circle, inner sep = 0.3ex] (v5) at (5+2, 2) {5};
    \node[rectangle, minimum size=.5cm, fill=red!20, draw = red,  thick, inner sep = 0.3ex] (y1) at (5-2, 4) {$y_1$\vspace*{1cm}};
     \node[rectangle, minimum size=.5cm, fill=red!20, draw = red,  thick, inner sep = 0.3ex] (y2) at (5-1, 4) {$y_2$};
      \node[rectangle, minimum size=.5cm, fill=red!20, draw = red,  thick, inner sep = 0.3ex] (y3) at (5+0, 4) {$y_3$};
       \node[rectangle, minimum size=.5cm, fill=red!20, draw = red,  thick, inner sep = 0.3ex] (y4) at (5+1, 4) {$y_4$};
         \node[rectangle, minimum size=.5cm, fill=red!20, draw = red,  thick, inner sep = 0.3ex] (y5) at (5+2, 4) {$y_5$};
    \draw [-{Stealth[length=3mm]}, shorten >=1mm, shorten <=1mm]  (central) -- (v1);
    \draw [-{Stealth[length=3mm]}, shorten >=1mm, shorten <=1mm] (central) -- (v2);
    \draw [-{Stealth[length=3mm]}, shorten >=1mm, shorten <=1mm]  (central) -- (v3);
    \draw [-{Stealth[length=3mm]}, shorten >=1mm, shorten <=1mm]  (central) -- (v4);
    \draw [-{Stealth[length=3mm]}, shorten >=1mm, shorten <=1mm]  (central) -- (v5);
    \draw [-{Stealth[length=3mm]}, shorten >=1mm, shorten <=1mm]  (y1) -- (v1);
    \draw [-{Stealth[length=3mm]}, shorten >=1mm, shorten <=1mm]  (y2) -- (v2);
    \draw [-{Stealth[length=3mm]}, shorten >=1mm, shorten <=1mm]  (y3) -- (v3);
    \draw [-{Stealth[length=3mm]}, shorten >=1mm, shorten <=1mm]  (y4) -- (v4);
     \draw [-{Stealth[length=3mm]}, shorten >=1mm, shorten <=1mm]  (y5) -- (v5);
\end{tikzpicture}
\end{minipage}
    \begin{rem}
    In \cite{KS24}, we build crepant resolutions of  not-necessarily conical, stratified varieties by gluing locally projective resolutions compatibly. Our construction can be used e.g., to build all 64 crepant resolutions for the four-pointed star in the previous subsection. But this construction cannot build $X$ in the five-pointed star example. 

    To our knowledge these give the first examples of proper crepant resolutions of conical symplectic singularities shown not to be given by iteratively taking locally projective crepant resolutions as in \cite{KS24}.
    \end{rem}
    
We give a generalization of this construction to quivers containing sufficiently many independent paths between vertices.

\subsection{Relation to the work of Arbo--Proudfoot \cite{AP16}} Since many of our  examples are toric hyperk\"ahler, they can (at least according to \cite[Conjecture 5.8]{AP16}) be expressed via zonotopes, but for nonprojective resolutions, especially when they are not given by iterative locally projective partial resolutions, it seems nontrivial to do so, as this is in a quite different spirit. Also, it is not clear to us how to interpret the condition on a zonotope to produce a quiver variety. It would be interesting to make explicit connections between our constructions and those of \cite{AP16}.
    
We emphasize, however, that our constructions can also be used to produce examples in the setting of quiver varieties which are \emph{not} toric hyperk\"ahler (i.e., have dimensions at vertices greater than one), see Example \ref{exam:nonhyptertoric}.

\subsection{Relation to the work of Hubbard \cite{Hubbard}}
We learned the method using good quotients of $G$-stable open subsets from Austin Hubbard, following his paper \cite{Hubbard}, though the combinatorics are different here and the examples apply to general quivers.  Note though that, unlike in \cite{Hubbard} in the setting of hyperpolygon spaces, we do not discuss full classifications of proper crepant resolutions, and content ourselves with producing interesting families of examples.  This is partly because the general combinatorial classification of proper crepant resolutions in \cite{Hubbard} using Cox ring constructions explained in \cite{ADHL} results in a somewhat inexplicit answer for a general quiver, and it appears to be less illuminating than giving explicit examples.  As an illustration of this, in \cite{Hubbard} a nice combinatorial description of the number of proper crepant resolutions of hyperpolygon cones is given, but for a general quiver, even the number of projective crepant resolutions, equal by \cite{BCS23} to the number of regions of certain hyperplane arrangement complements, is difficult to explicitly interpret combinatorially. 
       
\subsection{Framed vs unframed quivers} \label{ss:unframed}
In the preceding description we spoke of framed Nakajima quiver varieties, partly to make the combinatorial data arguably simpler.  It is convenient to replace these by unframed quiver varieties as done by Crawley--Boevey in \cite[middle of page 261]{CB01}:  Nakajima's quiver variety for a framed quiver $Q$ with vertex set $Q_0$ and framing $w \in \bN^{Q_0}$ is isomorphic to that for an unframed quiver, called $\tilde Q$, defined by adding a single additional vertex, called $\infty$, with dimension one, and for each $i \in Q_0$, adding $w_i$ arrows $\infty \to i$. That is, given the dimension vector $v \in \bN^{Q_0}$, the new dimension $\tilde v \in \bN^{\tilde Q_0}$ is given as $\tilde v = (v,1)$, for $1$ the dimension at $\infty$. Given a stability condition $\theta \in \bZ^{Q_0}$ we consider the unique $\tilde \theta \in \bZ^{\tilde Q_0}$ with the property that $\tilde \theta \cdot \tilde v = 0$ (the condition we always required in the unframed setting).

Under this operation, the $n$-pointed star with a framing of $1$ at each exterior vertex becomes a bipartite graph with $2+n$ vertices and $2n$ arrows: the original central node and the new vertex $\infty$ each have an arrow to each of the remaining $n$ vertices. The case $n=4$ is shown with the framed vertices depicted using squares, central vertex in blue, and the $\infty$ vertex in red:
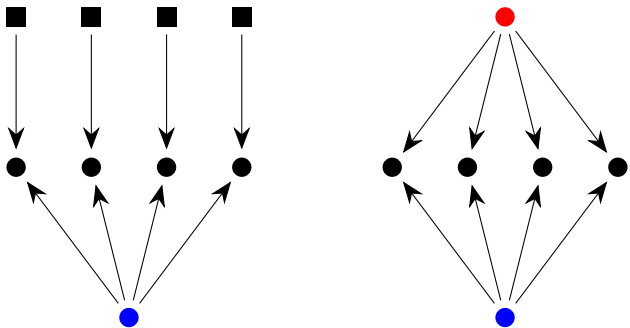
\[
\begin{tikzpicture}
    \node[circle, fill=blue, inner sep = 0.5ex] (central) at (0, 0) {};
    \node[circle, fill=black, inner sep = 0.5ex] (v1) at (-3/2, 2) {};
    \node[circle, fill=black, inner sep = 0.5ex] (v2) at (-1/2, 2) {};
    \node[circle, fill=black,inner sep = 0.5ex] (v3) at (1/2, 2) {};
    \node[circle, fill=black, inner sep = 0.5ex] (v4) at (3/2, 2) {};
    \node[rectangle, fill=black, inner sep = 0.7ex] (f1) at (-3/2, 4) {};
    \node[rectangle, fill=black, inner sep = 0.7ex] (f2) at (-1/2, 4) {};
    \node[rectangle, fill=black, inner sep = 0.7ex] (f3) at (1/2, 4) {};
    \node[rectangle, fill=black, inner sep = 0.7ex] (f4) at (3/2, 4) {};
    \draw [-{Stealth[length=3mm]}, shorten >=1mm, shorten <=1mm] (central) -- (v1);
    \draw [-{Stealth[length=3mm]}, shorten >=1mm, shorten <=1mm] (central) -- (v2);
    \draw [-{Stealth[length=3mm]}, shorten >=1mm, shorten <=1mm] (central) -- (v3);
    \draw [-{Stealth[length=3mm]}, shorten >=1mm, shorten <=1mm] (central) -- (v4);
    \draw [-{Stealth[length=3mm]}, shorten >=1mm, shorten <=1mm] (f1) -- (v1);
    \draw [-{Stealth[length=3mm]}, shorten >=1mm, shorten <=1mm] (f2) -- (v2);
    \draw [-{Stealth[length=3mm]}, shorten >=1mm, shorten <=1mm] (f3) -- (v3);
    \draw [-{Stealth[length=3mm]}, shorten >=1mm, shorten <=1mm] (f4) -- (v4);

    \node[circle, fill=blue, inner sep = 0.5ex] (central) at (5, 0) {};
    \node[circle, fill=black, inner sep = 0.5ex] (v1) at (5-3/2, 2) {};
    \node[circle, fill=black, inner sep = 0.5ex] (v2) at (5-1/2, 2) {};
    \node[circle, fill=black, inner sep = 0.5ex] (v3) at (5+1/2, 2) {};
    \node[circle, fill=black, inner sep = 0.5ex] (v4) at (5+3/2, 2) {};
    \node[circle, fill=red, inner sep = 0.5ex] (infinity) at (5, 4) {};
    \draw [-{Stealth[length=3mm]}, shorten >=1mm, shorten <=1mm] (central) -- (v1);
    \draw [-{Stealth[length=3mm]}, shorten >=1mm, shorten <=1mm] (central) -- (v2);
    \draw [-{Stealth[length=3mm]}, shorten >=1mm, shorten <=1mm] (central) -- (v3);
    \draw [-{Stealth[length=3mm]}, shorten >=1mm, shorten <=1mm] (central) -- (v4);
    \draw [-{Stealth[length=3mm]}, shorten >=1mm, shorten <=1mm] (infinity) -- (v1);
    \draw [-{Stealth[length=3mm]}, shorten >=1mm, shorten <=1mm] (infinity) -- (v2);
    \draw [-{Stealth[length=3mm]}, shorten >=1mm, shorten <=1mm] (infinity) -- (v3);
    \draw [-{Stealth[length=3mm]}, shorten >=1mm, shorten <=1mm] (infinity) -- (v4);
\end{tikzpicture}
\]
This way of thinking of the aforementioned examples, although equivalent, is more symmetric,  can be more convenient, and can give rise to additional generalizations (when the dimension vector at the framing vertex is allowed to increase).

\subsection*{Acknowledgments}
We would like to thank Austin Hubbard for patiently explaining the subtle combinatorics in his stimulating paper \cite{Hubbard}. We are grateful to Nick Proudfoot who pointed out the four-dimensional example (see Subsection \ref{ss:Proudfoot}). DK was supported by the European Research Council grant ERC-2019-ADG, the EPSRC Small Grant EP/Y033574/1, and thanks University of Hasselt, UMass Boston, and Queen Mary University London for the pleasant working conditions. 

\section{Preliminary results}
We work over the field $\bC$. We collect in this section some background information we will need.

\subsection{Nakajima quiver varieties}
\subsubsection{Definition}
Nakajima defined a class of moduli spaces of (semistable) $\bC$-representations of quivers satisfying so-called preprojective relations. The construction depends on a triple of combinatorial data: (1) a quiver, (2) a dimension vector, and (3) a stability parameter. 


A \emph{quiver} $Q = (Q_0, Q_1, s, t)$ is a directed graph with vertex set $Q_0$, arrow set $Q_1$, and source and target maps $s, t: Q_1 \to Q_0$ respectively. 
Given a quiver $Q$, its \emph{double quiver} $\overline{Q}$ has the same vertex set as $Q$ but with two arrows $a$ and $a^*$ for each arrow $a \in Q_1$, satisfying $s(a^*) = t(a)$ and $t(a^*) = s(a)$.

For example, 
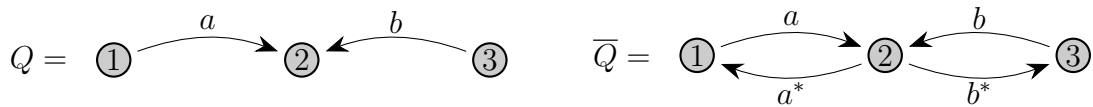
\[
\raisebox{.4cm}{
 \begin{tikzpicture}
    \node[] (Q) at (-1, 0) {$Q =$};
    \node[circle, fill=black!20, draw = black, thick, inner sep = 0.2ex] (v1) at (0, 0) {1};
    \node[circle, fill=black!20, draw = black, thick, inner sep = 0.2ex] (v2) at (2.5, 0) {2};
    \node[circle, fill=black!20, draw = black, thick, inner sep = 0.2ex] (v3) at (5, 0) {3};
    \node[] (a) at (1.25, 0.5) {$a$};
     \node[] (b) at (3.75, 0.5) {$b$};
    \draw [bend left=20, -{Stealth[length=3mm]}, shorten >=1mm, shorten <=1mm] (v1) to (v2);
    \draw [bend right=20, -{Stealth[length=3mm]}, shorten >=1mm, shorten <=1mm] (v3) to (v2);
\end{tikzpicture}} \hspace{1cm}
 \begin{tikzpicture} 
    \node[] (Q) at (-1, 0) {$\overline{Q} =$};
    \node[circle, fill=black!20, draw = black, thick, inner sep = 0.2ex] (v1) at (0, 0) {1};
    \node[circle, fill=black!20, draw = black, thick, inner sep = 0.2ex] (v2) at (2.5, 0) {2};
    \node[circle, fill=black!20, draw = black, thick, inner sep = 0.2ex] (v3) at (5, 0) {3};
    \draw [bend left=20, -{Stealth[length=3mm]}, shorten >=1mm, shorten <=1mm] (v2) to (v1);
      \draw [bend left=20, -{Stealth[length=3mm]}, shorten >=1mm, shorten <=1mm] (v1) to (v2);
    \draw [bend right=20, -{Stealth[length=3mm]}, shorten >=1mm, shorten <=1mm] (v2) to (v3);
     \draw [bend right=20, -{Stealth[length=3mm]}, shorten >=1mm, shorten <=1mm] (v3) to (v2);
      \node[] (a) at (1.25, 0.5) {$a$};
     \node[] (b) at (3.75, 0.5) {$b$};
     \node[] (da) at (1.25, -0.5) {$a^*$};
     \node[] (db) at (3.75, -0.5) {$b^*$};
\end{tikzpicture}
\]
 Pick a \emph{dimension vector} $\alpha = (\alpha_i) \in \bN^{Q_0}$. Define the vector space
\[
\Rep_\alpha(Q) := \bigoplus_{a \in Q_1} \Hom_{\bC}( \bC^{\alpha_{s(a)}}, \bC^{\alpha_{t(a)}} ).
\]
There is a natural action of $G_\alpha := \prod_{i \in Q_0} \GL_{\alpha_i}(\bC)$ on $\Rep_\alpha(Q)$ by $(g_i) \cdot (\rho_a) = (g_{t(a)} \circ \rho_a \circ g_{s(a)}^{-1})$. The diagonal $\bC^*$ acts trivially on $\Rep_\alpha(Q)$ so the action descends to $PG_\alpha := G_\alpha/ \bC^*$.

In $\Rep_\alpha(\overline{Q})$, one can consider the subset of representations:
\[
 \mu^{-1}_\alpha(0) := \left \{ \rho \in \Rep_\alpha(\overline{Q}) \ \ \middle | \ \ \sum_{a \in Q_1} \rho(a) \rho(a^*) - \rho(a^*) \rho(a) = 0 \right \}.
\]
We write $\mu^{-1}(0)$ for $\mu^{-1}_{\alpha}(0)$ when $\alpha$ is clear from context. The action of $PG_\alpha$ restricts to $\mu^{-1}(0)$.

Fix $\theta \in \bZ^{Q_0}$ satisfying $\theta \cdot \alpha = 0$, where $a \cdot b := \sum_{i} a_i b_i$ denotes the standard dot product. Notice that such a $\theta$ gives rise to a character of $G_\alpha$
\[
\chi_{\theta} : G_\alpha \to \bC^* \hspace{1cm} \chi_{\theta}( (g_i)_{i \in Q_0}) = \prod_{i \in Q_0} \det(g_i)^{\theta_i}
\]
that descends to a character of $PG_\alpha$ since $\theta \cdot \alpha = 0$. Conversely, all characters of $PG_{\alpha}$ are of the form $\chi_{\theta}$ for some $\theta \in \bZ^{Q_0}$ orthogonal to $\alpha$.  

Following King \cite[Definition 1.1, Section 3]{King}, a subrepresentation $W \subset V \in \Rep_\alpha(Q)$ is \emph{$\theta$-destabilizing} if $\dim(W) \cdot \theta > 0$ (where $\dim(W) = (\dim W_i)_{i \in Q_0} \in \bN^{Q_0}$ records the dimension at each vertex). And $V \in \Rep_\alpha(Q)$ is \emph{$\theta$-semistable} if no such $\theta$-destabilizing subrepresentation exists. $V$ is further \emph{$\theta$-stable} if its only subrepresentations $W$ satisfying $\dim(W) \cdot \theta =0$ are $W = V$ and $W = 0$. Given a subset $Z \subset \Rep_\alpha(Q)$, we denote by $Z^{\theta-\text{ss}}$ (resp. $Z^{\theta-\text{s}}$) the subset of $\theta$-semistable (resp. $\theta$-stable) representations in $Z$. A \emph{$\theta$-polystable} representation is a direct sum of $\theta$-stable representations.  The loci of $\theta$-semistable and $\theta$-stable representations equals the   $\chi_\theta$-semistable and $\chi_\theta$-stable loci, respectively, in the language of GIT.

If a group $G$ acts on an affine variety $V$ and $\chi: G \to \bC^\times$ is a character, then 
\[
\bC[V]^\chi := \{f \in \bC[V] \mid g \cdot f = \chi(g) f  \ \text{ for all } g \in G\}
\]
denotes the functions on which $G$ acts by $\chi$. 
\begin{defn}[\cite{Nakajima}] \label{defn:Nakajima}
    Fix a quiver $Q$, dimension vector $\alpha \in \bN^{Q_0}$, and a stability parameter $\theta \in \bZ^{Q_0}$.
    The \emph{Nakajima quiver variety}, $\cM_{\theta}(Q, \alpha)$, is the quotient 
    \[
    \cM_{\theta}(Q, \alpha) := \mu^{-1}(0) /\!/_{\chi_{\theta}} PG_{\alpha}
    := \Proj \left ( \bigoplus_{n \geq 0} \bC[ \mu^{-1}(0) ]^{\chi_{\theta}^n} \right )
    \]
\end{defn}
By the GIT construction, there is a good quotient $\mu^{-1}(0)^{\theta-ss} \to \cM_{\theta}(Q,\alpha)$ by $G_\alpha$ (see Section \ref{ss:good-quotient} below for the definition of a good quotient).
The Proj construction ensures $\cM_{\theta}(Q, \alpha)$ is projective over $\cM_{0}(Q, \alpha) = \Spec( \bC[\mu_{\alpha}^{-1}(0)]^{G_{\alpha}} )$. Note that $\cM_{0}(Q, \alpha)$ is conical with a $\bC^*$-action contracting to the zero representation. 

This paper will mostly consider the case $\alpha = (1, \dots, 1)$ so $G_{\alpha}$ is a torus and $\cM_{\theta}(Q, \alpha)$ is a toric hyperk\"{a}hler variety. In this setting the combinatorics dramatically simplifies, and yet many of our constructions extend to more general $\alpha$.  

\subsubsection{Local neighborhoods} \label{ss:local}
Let $V = V_1^{d_1} \oplus \dots \oplus V_m^{d_m}$ be a  $\theta$-polystable representation in $\cM_{\theta}(Q, \alpha)$. Following \cite[Corollary 4.10]{CB03} for $\theta=0$, and \cite[Theorem 3.3]{BS21} in general, $V$ has an \'{e}tale neighborhood $[V]_{\cM_{\theta}(Q, \alpha)}$ isomorphic to an \'{e}tale neighborhood of the zero representation $[0]_{\cM_{0}(Q^{\text{loc}}, d)}$ where $d = (d_i)$ and $Q^{\text{loc}}$ has the following combinatorial description.

The Euler bilinear form is defined by
\[
\langle -, - \rangle : \bZ^{Q_0} \times \bZ^{Q_0} \to \bZ \hspace{1cm} \langle \alpha, \beta \rangle = \sum_{i \in Q_0} \alpha_i \beta_i - \sum_{a \in Q_1} \alpha_{s(a)} \beta_{t(a)}.
\]
This form has symmetrization $( \alpha, \beta) := \langle \alpha, \beta \rangle + \langle \beta, \alpha \rangle$ and quadratic form $q(\alpha) := \frac{1}{2} (\alpha, \alpha)$. 

The local quiver $Q^{\text{loc}}$ at $V$ is defined, up to arbitrary choice of orientation, by the property of having double $\overline{Q}^{\text{loc}}$ with:
\begin{itemize}
\item a vertex $v_i$ for each representation $V_i$,
\item $\ell_i := 2-2q(\dim(V_i))$ loops at vertex $v_i$, and
\item $a_{i,j} := -(\dim(V_i), \dim(V_j))$ arrows from $v_i$ to $v_j$.
\end{itemize}

We observe that for many of our examples, the local description can be greatly simplified. First, if $\dim(V) = \alpha = (1, 1, \dots, 1)$ then (a) each $d_i = 1$, (b) the $V_i$ determine a collection of disjoint, connected subquivers of $\overline{Q}$ with Euler characteristic $q(\dim(V_i))$, 
 and (c) $a_{i, j}$ is the number of arrows in $\overline{Q}$ from the support of $V_i$ to the support of $V_j$. Note that, in many of our examples the subquivers corresponding to $V_i$ are supported on subquivers which are trees, so that  $\ell_i =0$ and there are no loops in the local quiver.

 Additionally, if there are only two vertices $x, y$ with $\theta_x , \theta_y  \geq 0$ then $V = V_x \oplus V_y$ (as each subrepresentation $V_i$ satisfies $\sum_{i \in \text{supp}(V_i)} \theta_i = \theta \cdot \dim(V_i) = 0$ and hence some $\theta_j \geq 0$). The local quiver is then: 
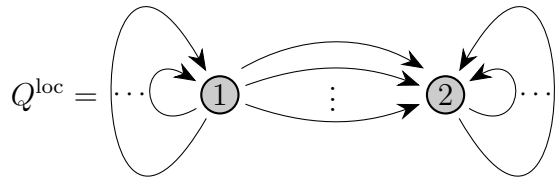
\vspace*{-2cm}
 \[
\begin{tikzpicture}
    \node[] (Q) at (-2.2, 0) {$Q^{\text{loc}}=$};
    \node[] (vdots) at (1.5, 0.05) {\vdots};
    \node[circle, fill=black!20, draw = black, thick, inner sep = 0.3ex] (v1) at (0, 0) {1};
    \node[circle, fill=black!20, draw = black, thick, inner sep = 0.3ex] (v2) at (3, 0) {2};

    \draw [bend left=30, -{Stealth[length=3mm]}, shorten >=1mm, shorten <=1mm] (0.2, 0.3) to (2.8,0.3);
    \draw [bend right=20, -{Stealth[length=3mm]}, shorten >=1mm, shorten <=1mm] (v1) to (v2);
    \draw [bend left=20, -{Stealth[length=3mm]}, shorten >=1mm, shorten <=1mm] (v1) to (v2);

    \draw [looseness=10, in=150, out=210, min distance=7mm, -{Stealth[length=3mm]}, shorten >=1mm, shorten <=1mm] (v1) to (v1);
    \draw [looseness=19.5, in=120, out=240, min distance=13mm, -{Stealth[length=3mm]}, shorten >=1mm, shorten <=1mm] (v1) to (v1);
    \node[] at (-1.2,0) {\scalebox{0.9}{${\bf \cdots}$}};

    \draw [looseness=10, in=30, out=-30, min distance=7mm, -{Stealth[length=3mm]}, shorten >=1mm, shorten <=1mm] (v2) to (v2);
    \draw [looseness=19.5, in=60, out=-60, min distance=13mm, -{Stealth[length=3mm]}, shorten >=1mm, shorten <=1mm] (v2) to (v2);
    \node[] at (4.2,0) {\scalebox{0.9}{${\bf \cdots}$}};
\end{tikzpicture}\vspace*{-2cm}
\]

If $a_{1, 2} > 2$, then an \'{e}tale neighborhood of the zero representation $[0]_{\cM_{0}(Q^{\text{loc}}, d)}$ has exactly two projective crepant resolutions given by variation of stability parameter: $\theta = (1, -1)$ or $-\theta = (-1, 1)$. 

If $Q^{\text{loc}}$ has no loops, then these resolutions can also be described using the parabolic Springer resolution of the minimal nonzero nilpotent orbit closure in $\mathfrak{sl}_n$ for  $n=a_{1,2}$:
\begin{align*}
&\cM_{0}(Q^{\text{loc}}, (1, 1)) \cong \{ M \in \Mat_{n \times n}(\bC) : M^2 = 0, \text{rank}(M) \leq 1 \} \\
&\cM_{(1, -1)}(Q^{\text{loc}}, (1, 1)) \cong T^*(\bP^{n-1}) \\
&\cM_{(-1, 1)}(Q^{\text{loc}}, (1, 1)) \cong T^*(\bP^{n-1})^{\vee}.
\end{align*}
This will serve as the local model for the majority of examples in this paper. Note that, when $a_{1,2}=2$, this description is still accurate, except that $T^*\bP^1 \cong T^*(\bP^1)^\vee$ as resolutions, both being the minimal resolution of the type $A_1$ Du Val singularity---the unique projective crepant resolution.

Each loop $\ell$ in $Q^{\text{loc}}$ contributes to the local model via a product by the symplectic affine space $\bA^2$. Therefore the addition of loops does not alter the local classification of crepant (partial) resolutions. 

\subsubsection{Resolutions of singularities and stratifications}

Under mild conditions, $\cM_{\theta}(Q,\alpha) \to \cM_0(Q,\alpha)$ is a projective crepant partial resolution of singularities for all $\theta$, and conversely, all projective crepant partial resolutions are of this form:
\begin{thm} \cite[Corollary 4.7]{BCS23} \label{thm:BCS} If $\alpha$ is a dimension vector with some $\alpha_i = 1$, such that there exists a simple representation in $\mu^{-1}(0)$, then the projective crepant partial resolutions of $\cM_{0}(Q, \alpha)$ are precisely the maps $\cM_{\theta}(Q, \alpha) \to \cM_0(Q,\alpha)$ for $\theta \in \bZ^{Q_0}$.
\end{thm}
Although we do not require it, in \cite{BCS23}, it is also explained when $\theta,\theta'$ define isomorphic partial resolutions (namely, if the GIT cones containing them have the same orbit under the Namikawa--Weyl group).

Next, to construct \emph{locally projective} crepant resolutions of $\cM_{\theta}(Q, \alpha)$ we use the approach of \cite{KS24} relying on the natural Luna stratification \cite[\S 6]{Nakajima}, \cite[\S 3.v]{Nakajima98}, which is also the stratification by symplectic leaves:

 \begin{thm} 
 \cite[Theorem 1.9]{BS21} \label{thm:BS}
 $\cM_{\theta}(Q, \alpha)$ has a finite stratification by symplectic leaves. Each symplectic leaf is given by a decomposition of $\alpha$ into summands orthogonal to $\theta$. 
 \end{thm}


We do not require the results of \cite{KS24}, because instead we define the locally projective crepant resolutions explicitly via quotients $U/\!/G_{\alpha}$ of $G_{\alpha}$-stable open sets $U \subseteq \mu_\alpha^{-1}(0)$.

We show these resolutions are not projective by appealing to Theorem \ref{thm:BCS} and showing they are not of the form $\cM_{\theta}(Q, \alpha)$ for some $\theta$. 

We additionally utilize foundational results (see \cite{Drezet-Luna-slice}) to show that $U/\!/G_{\alpha}$ is a geometric quotient. We recall the necessary background in the next subsection.

\subsection{Good quotients}\label{ss:good-quotient}

We recall the following standard definition (see, e.g., \cite{Drezet-Luna-slice}):
\begin{defn} Let $Y$ and $Z$ be varieties.
    A \emph{good} quotient $\pi: Y \to Z =: Y/\!/G$ is a $G$-invariant, affine, surjective morphism such that, for every open subset $U \subseteq Z$, the natural pullback map,
\[
\Gamma(U,\mathcal{O}_U) \to \Gamma(\pi^{-1}(U), \mathcal{O}_{\pi^{-1}(U)}),
\]
is an isomorphism, and for $W_1, W_2 \subseteq Y$ two $G$-stable closed subsets, $\pi(W_1)$ and $\pi(W_2)$ are disjoint.
\end{defn} 
Note that good quotients are also categorical quotients, so in particular they are canonical when they exist.

We collect the following basic results:
\begin{lem}\label{l:good-quotient}
\cite[Lemma 2.13]{Drezet-Luna-slice} If $\pi$ is a good quotient, then it induces a bijection between closed $G$-orbits and closed points of the target.
\end{lem}
\begin{thm} \label{t:hilbert} \cite[Theorem 3.5]{Newstead},\cite{GIT}, \text{cf.~} \cite[Theorem 2.16]{Drezet-Luna-slice}
If $Y$ is an affine variety over $\bC$ and $G$ is complex reductive, then $Y/\!/G := \Spec \mathcal{O}(Y)^G$ is a good quotient.
\end{thm}
\begin{defn}
A \emph{geometric} quotient is a good quotient such that $\pi(x) \neq \pi(y)$ whenever $G \cdot x \neq G \cdot y$.
\end{defn}
It follows from Lemma \ref{l:good-quotient} that a good quotient is geometric if and only if all $G$-orbits on $Y$ are closed.

\section{Locally projective constructions}\label{s:loc-proj}
    \subsection{Four-pointed star example} \label{ss:four-pointed}
    \begin{minipage}{0.7\textwidth}
    \setlength{\parindent}{15pt}
    \noindent Consider the four-pointed star with all dimensions one and a framing of one at each vertex.  For symmetry, it is more convenient for us to use the equivalent unframed formulation: $Q_0=\{x,1,2,3,4,y\}$, with $Q_1 = \{ \phi^x_i, \phi^y_i \}_{i = 1, \dots, 4}$ consisting of one arrow $\phi^x_i$ (resp. $\phi^y_i$) from $x$ (resp. $y$) to $i \in \{ 1, \ldots, 4 \}$. Thus, there are a total of eight arrows (shown right). 
    
     \indent Taking the dimension vector $\alpha = (1,1,1,1,1,1)$ and stability parameter $(0, 0, 0, 0, 0, 0)$ (i.e., no stability condition) yields the conical quiver variety $\cM_0(Q,\alpha)$.

    \indent The stability parameter $\theta = (2,-1,-1,-1,-1,2)$ gives a projective crepant partial resolution \vspace{-.3cm}
    \[
    \pi: \cM_\theta(Q,\alpha) \to \cM_0(Q,\alpha). \\
    \] 
    \end{minipage} \hspace{0.5cm} 
    \begin{minipage}{0.15\textwidth}
    \vspace*{-2cm}
    
   \[
\begin{tikzpicture}
    \node[circle, fill=blue!20, draw = blue,  thick, circle, inner sep = 0.3ex] (central) at (0, 0) {$x$};
    \node[circle, fill=black!20, draw = black, thick, circle, inner sep = 0.3ex] (v1) at (-3/2, 2.5) {1};
    \node[circle, fill=black!20, draw = black, thick, circle, inner sep = 0.3ex] (v2) at (-1/2, 2.5) {2};
    \node[circle, fill=black!20, draw = black, thick, circle, inner sep = 0.3ex] (v3) at (1/2, 2.5) {3};
    \node[circle, fill=black!20, draw = black, thick, circle, inner sep = 0.3ex] (v4) at (3/2, 2.5) {4};
    \node[circle, fill=red!20, draw = red, thick, circle, inner sep = 0.3ex] (infinity) at (0, 5) {$y$};
    \draw [-{Stealth[length=3mm]}, shorten >=1mm, shorten <=1mm] (central) -- (v1);
    \draw [-{Stealth[length=3mm]}, shorten >=1mm, shorten <=1mm] (central) -- (v2);
    \draw [-{Stealth[length=3mm]}, shorten >=1mm, shorten <=1mm] (central) -- (v3);
    \draw [-{Stealth[length=3mm]}, shorten >=1mm, shorten <=1mm] (central) -- (v4);
    \draw [-{Stealth[length=3mm]}, shorten >=1mm, shorten <=1mm] (infinity) -- (v1);
    \draw [-{Stealth[length=3mm]}, shorten >=1mm, shorten <=1mm] (infinity) -- (v2);
    \draw [-{Stealth[length=3mm]}, shorten >=1mm, shorten <=1mm] (infinity) -- (v3);
    \draw [-{Stealth[length=3mm]}, shorten >=1mm, shorten <=1mm] (infinity) -- (v4);
\end{tikzpicture}
   \]
   \end{minipage}
     \vspace{.3cm}
     
    The space $\cM_\theta(Q,\alpha)$ is a symplectic singularity, and hence is stratified by symplectic leaves \cite[Theorem 2.3]{Kaledin_Poisson}. In this case, each leaf is given by a decomposition of the dimension vector into summands orthogonal to $\theta$, see \cite[Theorem 1.9]{BS21}. These decompositions are:
    \begin{itemize}
        \item trivial $\alpha = \alpha$, corresponding to the open symplectic leaf;
        \item of the form $\alpha = \alpha^x + \alpha^y$, where $\alpha^x = e_x + e_{k} + e_{\ell}$ for some $1 \leq k < \ell \leq 4$, and $\alpha^y = \alpha-\alpha^x=e_y + e_{k'} + e_{\ell'}$ for $\{k,\ell,k',\ell'\} = \{1,2,3,4\}$. This yields ${4 \choose 2} = 6$ closed symplectic leaves.
    \end{itemize}
    As explained in Subsection \ref{ss:local}, each closed symplectic leaf has local singularity described by a local quiver with two vertices (one for each summand) and four arrows between these vertices.  This singularity is well known to be isomorphic to the locus of square-zero, rank-one matrices of size $4\times 4$
    (i.e., the minimal nonzero nilpotent orbit closure in $\mathfrak{sl}_4$), which admits exactly two nonisomorphic projective crepant resolutions: each given by variation of GIT quotient for stability parameters $\{ \pm \theta \}$ where $\theta := (a,-a)$ with $a > 0$.
    
    Since these leaves are isolated singularities, locally projective crepant resolutions are uniquely determined by an arbitrary choice of projective resolution of small neighborhoods of each singularity\footnote{This observation is well-known but also Corollary 3.39 of \cite{KS24}.}.
    Therefore, there are precisely $2^6=64$ locally projective crepant resolutions of $\cM_\theta(Q,\alpha)$. We will explicitly construct these $64$ resolutions.
    
    Let $\Phi = \{e_x+e_i+e_j: 1 \leq i < j \leq 4\}$ be the set of roots containing $e_x$ orthogonal to $\theta$.  We will label the $64$ resolutions by maps $s: \Phi \to \{+, -\}$ giving a sign to each element of $\Phi$.   
    
    \begin{defn}\label{d:Us}
        Given $s: \Phi \to \{+, -\}$ we define the open subset $U_s \subseteq \mu^{-1}(0)^{\theta-\text{ss}}$ of all $\theta$-semistable representations such that, for every $\beta \in \Phi$, if $s(\beta)=+$, then there exists a nonzero arrow from a vertex in the support of $\beta$ to one in the support of $\alpha-\beta$; and if $s(\beta)=-$, there exists a nonzero arrow in the opposite direction. 
    \end{defn}
    
    Then we see that $U_s /\!/ G_\alpha \to \mu^{-1}(0)^{\theta-\text{ss}} /\!/ G_\alpha$ is locally given over a point of a closed symplectic leaf as a projective crepant resolution, with isomorphism class depending precisely on the value of $s$ on the root contained in its decomposition.  Thus, we get $64$ locally projective crepant resolutions. Since this is proper and crepant, by composition $U_s /\!/ G_\alpha \to \cM_0(Q,\alpha)$ is proper and crepant. Further, it is a resolution since the source is smooth.  

For convenience of notation, let $\beta_{i,j} := e_x + e_i+e_j$.
\begin{thm}\label{thm:four-point} Not all of the $64$ locally projective crepant resolutions of $\cM_\theta(Q,\alpha)$ are projective over $\cM_0(Q,\alpha)$.  In particular, if for some partition $\{1,2,3,4\} = \{i,j\} \sqcup \{k,\ell\}$, we have:
\begin{gather}
s(\beta_{i,j}) = s(\beta_{k,\ell}) = +, \label{n=4-cond1} \\
s(\beta_{i,k}) = s(\beta_{j,\ell}) = -, \label{n=4-cond2}
\end{gather}
then the resolution $U_s / \!/ G_{\alpha}$ is not projective over $\cM_0(Q,\alpha)$.
\end{thm}
\begin{proof} Thanks to \cite{BCS23}, all projective crepant resolutions of $\cM_0(Q,\alpha)$ are given by variation of GIT quotient, i.e., are of the form $\cM_{\theta'}(Q,\alpha)$. Moreover, such a map  factors through $\cM_\theta(Q,\alpha)$ if and only if $\theta$ is in the boundary of the GIT chamber containing $\theta'$. 

Given such a $\theta'$, we have $\mu_{\alpha}^{-1}(0)^{\theta'-ss}=U_{s_{\theta'}}$,
where $s_{\theta'}: \Phi \to \{+,-\}$ is determined by $s_{\theta'}(\beta) = \sgn(\theta' \cdot \beta)$. 

By construction, the resolutions $U_{s_{\theta'}}/\!/G_\alpha$ are pairwise nonisomorphic $\cM_\theta(Q,\alpha)$. As a result they are also nonisomorphic over $\cM_0(Q,\alpha)$, since the map $\cM_\theta(Q,\alpha) \to \cM_0(Q,\alpha)$ is birational, so the condition for an isomorphism to induce the identity on $\cM_\theta(Q,\alpha)$ is equivalent to inducing the identity over $\cM_0(Q,\alpha)$. 


Now let $s: \Phi \to \{+,-\}$ satisfy \eqref{n=4-cond1},\eqref{n=4-cond2}. 
To show that there is no $\theta'$ such that $\cM_{\theta'}(Q,\alpha) \cong U_s/\!/G_\alpha$ over $\cM_0(Q,\alpha)$, it suffices to show that $s \neq s_{\theta'}$ for any stability condition $\theta'$. 

Suppose $s_{\theta'}$ satisfies \eqref{n=4-cond1}, \eqref{n=4-cond2}. Then additivity of the dot product implies:
\begin{align*}
\theta' \cdot \beta_{i,j}, \ \ \theta' \cdot \beta_{k,\ell} > 0
&\implies \theta' \cdot (\beta_{i,j} +\beta_{k,\ell}) > 0 \\
\theta' \cdot \beta_{i,k}, \ \ \theta' \cdot \beta_{j,\ell} < 0 
&\implies \theta' \cdot (\beta_{i,k} +\beta_{j,\ell}) < 0.
\end{align*}
But $\beta_{i, j} + \beta_{k, \ell} = \beta_{i, k} + \beta_{j, \ell}$ so
\[
\theta' \cdot (\beta_{i, j} + \beta_{k, \ell}) = \theta' \cdot (\beta_{i, k} + \beta_{j, \ell})
\]
a contradiction. 
\end{proof}

\begin{rem}\label{r:18-resolutions}
    For each (2, 2)-partition of $\{ 1, 2, 3, 4\}$, there are $4+4-1=7$ choices of resolution satisfying \eqref{n=4-cond1},\eqref{n=4-cond2}. Each pair of partitions has exactly one of these in common, which is not in common with the other partition. So in total, Theorem \ref{thm:four-point} shows that $3\cdot 7 - 3 = 18$ of the $64$ locally projective crepant resolutions of $\cM_\theta(Q,\alpha)$ define nonprojective crepant resolutions of  $\cM_0(Q,\alpha)$. The remaining 46 are projective over $\cM_0(Q,\alpha)$, as we verified using a computer.
\end{rem}

\subsection{Other graphs with six vertices}
Using the same proof as for Theorem \ref{thm:four-point} we can establish the same result for other unframed quivers with six vertices.  These can also be viewed as framed quivers with five vertices, as in Section \ref{ss:unframed}, by deleting one node and replacing it by framings at the adjacent nodes. We found these graphs by looking for the minimal ones for which the argument of Theorem \ref{thm:four-point} applies. Note that Remark \ref{r:add-arrows} shows that adding arrows does not affect the validity of the construction.
\begin{cor}{}
  The following graphs can be oriented arbitrarily to give a quiver $Q$ whose quiver variety $\cM_{0}(Q, \alpha)$ has a proper, nonprojective crepant resolution. It can be constructed by taking locally projective resolutions of $\cM_\theta(Q,\alpha)$ with $\theta=(2,-1,-1,-1,-1,2)$. 
\end{cor}
 
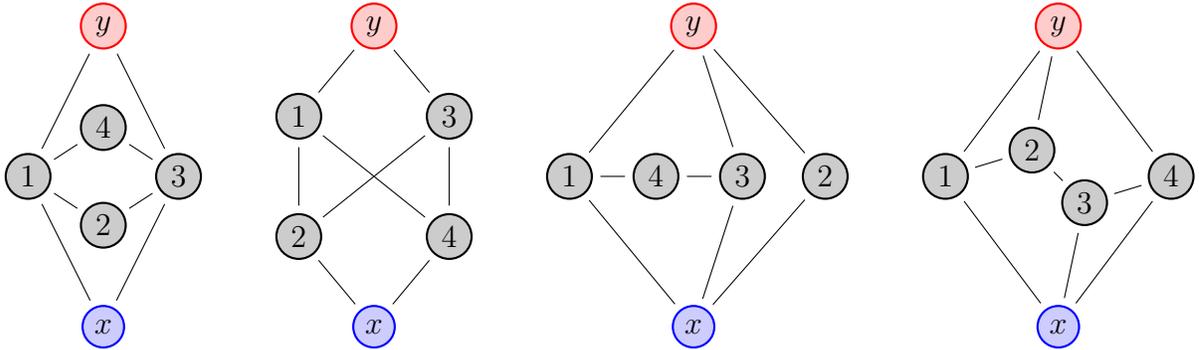
\begin{tikzpicture}
    \node[circle, fill=blue!20, draw = blue,  thick, circle, inner sep = 0.5ex] (x) at (0, -2) {$x$};
    \node[circle, fill=black!20, draw = black, thick, circle, inner sep = 0.5ex] (v1) at (-1, 0) {1};
    \node[circle, fill=black!20, draw = black, thick, circle, inner sep = 0.5ex] (v2) at (0, -.65) {2};
    \node[circle, fill=black!20, draw = black, thick, circle, inner sep = 0.5ex] (v3) at (1, 0) {3};
    \node[circle, fill=black!20, draw = black,  thick, circle, inner sep = 0.5ex] (v4) at (0, 0.65) {4};
    \node[circle, fill=red!20, draw = red,  thick, circle, inner sep = 0.5ex] (y) at (0, 2) {$y$};
    \draw [-, shorten >=1mm, shorten <=1mm] (x) -- (v1);
    \draw [-, shorten >=1mm, shorten <=1mm] (x) -- (v3);
    \draw [-, shorten >=1mm, shorten <=1mm] (v1) -- (v2);
    \draw [-, shorten >=1mm, shorten <=1mm] (v1) -- (v4);
    \draw [-, shorten >=1mm, shorten <=1mm] (v2) -- (v3);
    \draw [-, shorten >=1mm, shorten <=1mm] (v3) -- (v4);
    \draw [-, shorten >=1mm, shorten <=1mm] (y) -- (v1);
    \draw [-, shorten >=1mm, shorten <=1mm] (y) -- (v3);
\end{tikzpicture}  \hspace{.7cm}
\begin{tikzpicture}
    \node[circle, fill=blue!20, draw = blue,  thick, circle, inner sep = 0.5ex] (x) at (0, -2) {$x$};
    \node[circle, fill=black!20, draw = black, thick, circle, inner sep = 0.5ex] (v1) at (-1, 4/5) {1};
    \node[circle, fill=black!20, draw = black, thick, circle, inner sep = 0.5ex] (v2) at (-1, -4/5) {2};
    \node[circle, fill=black!20, draw = black, thick, circle, inner sep = 0.5ex] (v3) at (1, 4/5) {3};
    \node[circle, fill=black!20, draw = black,  thick, circle, inner sep = 0.5ex] (v4) at (1, -4/5) {4};
    \node[circle, fill=red!20, draw = red,  thick, circle, inner sep = 0.5ex] (y) at (0, 2) {$y$};
    \draw [-, shorten >=1mm, shorten <=1mm] (x) -- (v2);
    \draw [-, shorten >=1mm, shorten <=1mm] (x) -- (v4);
    \draw [-, shorten >=1mm, shorten <=1mm] (v1) -- (v2);
    \draw [-, shorten >=1mm, shorten <=1mm] (v1) -- (v4);
    \draw [-, shorten >=1mm, shorten <=1mm] (v2) -- (v3);
    \draw [-, shorten >=1mm, shorten <=1mm] (v3) -- (v4);
    \draw [-, shorten >=1mm, shorten <=1mm] (y) -- (v1);
    \draw [-, shorten >=1mm, shorten <=1mm] (y) -- (v3);
\end{tikzpicture} \hspace{.7cm}
\begin{tikzpicture}
    \node[circle, fill=blue!20, draw = blue,  thick, circle, inner sep = 0.5ex] (central) at (5, 0) {$x$};
    \node[circle, fill=black!20, draw = black, thick, circle, inner sep = 0.5ex] (v1) at (5-1.65, 2) {1};
    \node[circle, fill=black!20, draw = black, thick, circle, inner sep = 0.5ex] (v4) at (5-1/2, 2) {4};
    \node[circle, fill=black!20, draw = black, thick, circle, inner sep = 0.5ex] (v3) at (5+.65, 2) {3};
    \node[circle, fill=black!20, draw = black,  thick, circle, inner sep = 0.5ex] (v2) at (5+1.75, 2) {2};
    \node[circle, fill=red!20, draw = red,  thick, circle, inner sep = 0.5ex] (infinity) at (5, 4) {$y$};
    \draw [-, shorten >=1mm, shorten <=1mm] (central) -- (v1);
    \draw [-, shorten >=1mm, shorten <=1mm] (central) -- (v2);
    \draw [-, shorten >=1mm, shorten <=1mm] (central) -- (v3);
    \draw [-, shorten >=1mm, shorten <=1mm] (v1) -- (v4);
    \draw [-, shorten >=1mm, shorten <=1mm] (infinity) -- (v1);
    \draw [-, shorten >=1mm, shorten <=1mm] (infinity) -- (v2);
    \draw [-, shorten >=1mm, shorten <=1mm] (infinity) -- (v3);
    \draw [-, shorten >=1mm, shorten <=1mm] (v3) -- (v4);
\end{tikzpicture}  \hspace{.7cm}
\begin{tikzpicture}
    \node[circle, fill=blue!20, draw = blue,  thick, circle, inner sep = 0.5ex] (central) at (5, 0) {$x$};
    \node[circle, fill=black!20, draw = black, thick, circle, inner sep = 0.5ex] (v1) at (5-3/2, 2) {1};
    \node[circle, fill=black!20, draw = black, thick, circle, inner sep = 0.5ex] (v2) at (5-.35, 2.35) {2};
    \node[circle, fill=black!20, draw = black, thick, circle, inner sep = 0.5ex] (v3) at (5+.35, 2-.35) {3};
    \node[circle, fill=black!20, draw = black,  thick, circle, inner sep = 0.5ex] (v4) at (5+3/2, 2) {4};
    \node[circle, fill=red!20, draw = red,  thick, circle, inner sep = 0.5ex] (infinity) at (5, 4) {$y$};
    \draw [-, shorten >=1mm, shorten <=1mm] (central) -- (v1);
    \draw [-, shorten >=1mm, shorten <=1mm] (central) -- (v3);
    \draw [-, shorten >=1mm, shorten <=1mm] (v2) -- (v3);
    \draw [-, shorten >=1mm, shorten <=1mm] (central) -- (v4);
    \draw [-, shorten >=1mm, shorten <=1mm] (infinity) -- (v1);
    \draw [-, shorten >=1mm, shorten <=1mm] (infinity) -- (v2);
    \draw [-, shorten >=1mm, shorten <=1mm] (v1) -- (v2);
    \draw [-, shorten >=1mm, shorten <=1mm] (v3) -- (v4);
    \draw [-, shorten >=1mm, shorten <=1mm] (infinity) -- (v4);
\end{tikzpicture}
 
\begin{proof}
    Fix the partition $\{ 1, 2, 3, 4 \} = \{ 1, 2 \} \sqcup \{ 3, 4 \}$. Notice each graph above has a collection of eight connected subgraphs with vertices
\begin{align*}
&\{ x, 1, 2 \}, &&\{ x, 3, 4 \}, &&\{ x, 1, 4 \}, &&\{ x, 2, 3 \}, \\
&\{ y, 1, 2 \}, &&\{ y, 3, 4 \}, &&\{ y, 1, 4 \}, &&\{ y, 2, 3 \}.
\end{align*}
    Hence each set above has a subrepresentation supported on it. We can define $s$ satisfying  
    \[
    s( e_x + e_1 +e_2) = + = s(e_x + e_3 +e_4), \hspace{1cm} 
    s( e_x + e_1 +e_4) = - = s(e_x + e_2 +e_3)
    \]
    yielding the proper crepant resolution $U_{s} /\!/ G_{\alpha}$.  The argument given in Theorem \ref{thm:four-point} applies verbatim to see that this resolution is not projective.
\end{proof}
\begin{rem}
    Following Remark \ref{r:18-resolutions}, we can count precisely how many locally projective resolutions of $\cM_\theta(Q,\alpha)$ are nonprojective over $\cM_0(Q,\alpha)$. For the first and third graphs depicted, there are four decompositions corresponding to four isolated singularities, so $16$ locally projective resolutions, of which two are nonprojective, and the other $14$ are projective as a consequence of Remark \ref{r:18-resolutions}.  For the second and fourth graphs depicted, there is a fifth decomposition which does not affect the nonprojectivity and we get $4$ nonprojective resolutions out of $32$. 
    \end{rem}
\begin{rem}\label{r:add-arrows} Every result in this subsection holds if we add additional arrows to the quiver (see Corollary \ref{c:extend-resolutions}). This includes any unframed quiver on $6$ vertices containing the above complete bipartite one, and all dimensions $1$. By the correspondence from framed to unframed quivers this also includes all framed quivers on $5$ vertices containing a four-pointed star and all dimensions 1. The external framings need to be at least one, with an arbitrary framing at the central vertex. 

The only difference with the four-pointed star example is that the neighborhoods of points at the $6$ symplectic leaves will now be described by a quiver on $2$ vertices with $m \geq 4$ arrows between them (and loops at the vertices), still with dimension vector $(1,1)$. The singularity type is now isomorphic to the locus of square-zero, rank-one matrices of size $m \times m$, still admitting exactly two projective crepant resolutions as above. However, now the singularities need no longer be isolated, so \emph{a priori} there could be a monodromy obstruction (as in \cite{KS24}) to extending a choice of local resolution at a singular point along the whole stratum containing it.  This obstruction vanishes thanks to the explicit construction of resolution we give here; alternatively, in work in progress, we will study this obstruction for general quiver varieties, and we will see that it vanishes in this case because the decompositions of the dimension vector here have distinct summands. 
\end{rem}

\subsection{A four-dimensional example} \label{ss:Proudfoot}

After seeing the four-pointed example in Section \ref{ss:four-pointed}, Nick Proudfoot showed us the following four-dimensional example, equivalent to \cite[Figure 2]{AP16}. This is the minimal dimension of a symplectic resolution of a cone which is not projective, since in the two-dimensional case one only has du Val singularities, whose minimal resolutions are projective and crepant.

\noindent
\begin{minipage}{0.75\textwidth}
  \setlength{\parindent}{15pt}
  
\indent Let $Q$ be the quiver shown with $Q_0=\{x,1,2,3,4,5,6,y\}$ and arrows $x \to 1,2,3$, $y \to 4,5,6$, and $i \to i+3$ for $i = 1, 2, 3$. By Corollary \ref{c:extend-resolutions}, we can again consider enlargements with more arrows.

\indent Consider the quiver variety $\cM_{\theta}(Q, \alpha)$ with $\alpha=(1,1,1,1,1,1,1,1)$ and stability parameter 
$\theta = (3,-1,-1,-1,-1,-1,-1,3)$.

\indent Equivalently, as a framed quiver, we consider  a star with three legs of length two, and framings of one on each of the three endpoints. The framed stability parameter is $\theta = (3,-1,-1,-1,-1,-1,-1)$. 

\indent Then the non-open leaves of $\cM_\theta(Q,\alpha)$ for $\alpha=(1,1,1,1,1,1,1,1)$ correspond to the decompositions $\alpha = \beta + (\alpha-\beta)$ where 
\[
\beta \in \Phi := \{\beta_{i,i+3,j}: i,j \in \{1,2,3\}: i \neq j\} \cup \{\beta_{1,2,3}\}, \quad \beta_{i,j,k} := e_x+e_i+e_j+e_k.
\]
\end{minipage} \hspace{0.5cm}
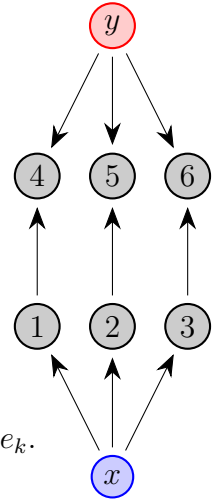
\begin{minipage}{0.1\textwidth}
\[
\begin{tikzpicture}
    \node[circle, fill=blue!20, draw = blue,  thick, circle, inner sep = 0.5ex] (central) at (0, 0) {$x$};
    \node[circle, fill=black!20, draw = black,  thick, circle, inner sep = 0.5ex] (v1) at (-1, 2) {1};
    \node[circle, fill=black!20, draw = black, thick, circle, inner sep = 0.5ex] (v2) at (0, 2) {2};
    \node[circle, fill=black!20, draw = black, thick, circle, inner sep = 0.5ex] (v3) at (1, 2) {3};
    \node[circle, fill=black!20, draw = black,  thick, circle, inner sep = 0.5ex] (v4) at (-1, 4) {4};
     \node[circle, fill=black!20, draw = black,  thick, circle, inner sep = 0.5ex] (v5) at (0, 4) {5};
      \node[circle, fill=black!20, draw = black,  thick, circle, inner sep = 0.5ex] (v6) at (1, 4) {6};
    \node[circle, fill=red!20, draw = red, thick, circle, inner sep = 0.5ex] (infinity) at (0, 6) {$y$};
    \draw [-{Stealth[length=3mm]}, shorten >=1mm, shorten <=1mm] (central) -- (v1);
    \draw [-{Stealth[length=3mm]}, shorten >=1mm, shorten <=1mm] (central) -- (v2);
    \draw [-{Stealth[length=3mm]}, shorten >=1mm, shorten <=1mm] (central) -- (v3);
    \draw [-{Stealth[length=3mm]}, shorten >=1mm, shorten <=1mm] (v1) -- (v4);
    \draw [-{Stealth[length=3mm]}, shorten >=1mm, shorten <=1mm] (v2) -- (v5);
    \draw [-{Stealth[length=3mm]}, shorten >=1mm, shorten <=1mm] (v3) -- (v6);
    \draw [-{Stealth[length=3mm]}, shorten >=1mm, shorten <=1mm] (infinity) -- (v4);
    \draw [-{Stealth[length=3mm]}, shorten >=1mm, shorten <=1mm] (infinity) -- (v5);
    \draw [-{Stealth[length=3mm]}, shorten >=1mm, shorten <=1mm] (infinity) -- (v6);
\end{tikzpicture}
\]
\end{minipage} 

Note that $|\Phi|=7$, and the elements each define zero-dimensional leaves. \\
In other words, $\cM_\theta(Q,\alpha)$ has seven isolated singularities.

Each singularity corresponds to a local quiver with two vertices connected by three arrows. Therefore, the singularity is isomorphic to square-zero $3 \times 3$ matrices of rank at most one.  Each singularity has two projective crepant resolutions: $T^* \mathbb{P}^2$ and $T^* (\mathbb{P}^2)^\vee$. Hence, there are $2^7=128$ different choices of local crepant resolutions, which all glue to locally projective crepant resolution.  As before, we label such resolutions by functions $s: \Phi \to \{+,-\}$, each yielding the resolution $U_s/\!/G_\alpha$ given by Definition \ref{d:Us}.
\begin{thm}
For the functions $s$ satisfying the following condition, the locally projective resolution $U_s/\!/G_\alpha \to \mathcal{M}_\theta(Q,\alpha)$ is nonprojective over $\cM_0(Q,\alpha)$:
\[
s(\beta_{i,i+3,j}) = \begin{cases} +, \quad \text{if $j\equiv i+1 \pmod 3$}, \\
-, \quad \text{if $j \equiv i-1 \pmod 3$}.
\end{cases}
\]
The same holds for the opposite signs $-s$.
\end{thm}

\begin{rem}
Note that $s(\beta_{1,2,3})$ is arbitrary in the theorem, so there are a total of four functions $s$ and hence four nonprojective resolutions $U_{s} /\!/G_{\alpha}$. A computer calculation revealed that the remaining 124 crepant resolutions are projective.
\end{rem}

\begin{proof}
The proof is the same as for Theorem \ref{thm:four-point}, except that we need to give a different proof that there is no $\theta'$ such that the sign of $\theta' \cdot \beta$ equals $s(\beta)$ for all $\beta \in \Phi$. 
Observe that
\[
\beta_{1,4,2}+\beta_{2,5,3}+\beta_{3,6,1} = \beta_{1,4,3}+\beta_{2,5,1}+\beta_{3,6,2}.
\]
There does not exist $\theta'$ with positive dot product with the left-hand side and negative dot product with the right-hand side, and hence $U_s/\!/G_\alpha \to \cM_0(Q,\alpha)$ cannot be projective.
\end{proof}

\subsection{Quivers with fewer vertices and higher dimensions}\label{ss:fewer-vertices}
We now consider examples where $|Q_0|=3$ and $4$, at the price of increasing dimensions at vertices (so some $\alpha_v > 1$).  
These are the first examples we are aware of in the literature with $|Q_0| < 6$. The three vertex example is minimal in the sense that if $|Q_0| < 3$ then any locally projective crepant resolution of $\cM_{\theta}(Q, \alpha)$ is projective.

\subsubsection{Three Vertex Example}
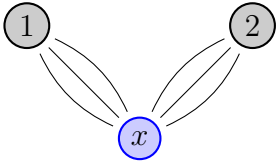
\begin{minipage}{0.2\textwidth}
\[
\raisebox{.7cm}{
\begin{tikzpicture}
    \node[circle, fill=blue!20, draw = blue, thick, inner sep = 0.5ex] (central) at (0, -1.5) {$x$};
    \node[circle, fill=black!20, draw = black, thick, inner sep = 0.5ex] (v1) at (-1.5, 0) {1};
    \node[circle, fill=black!20, draw = black, thick, inner sep = 0.5ex] (v2) at (1.5, 0) {2};
    \draw [-, shorten >=1mm, shorten <=1mm] (central) to (v1);
    \draw [-, shorten >=1mm, shorten <=1mm] (central) to (v2);
    \draw [-, shorten >=1mm, shorten <=1mm, bend right = 20] (central) to (v1);
    \draw [-, shorten >=1mm, shorten <=1mm, bend left = 20] (central) to (v2);
    \draw [-, shorten >=1mm, shorten <=1mm, bend left=20] (central) to (v1);
    \draw [-, shorten >=1mm, shorten <=1mm, bend right=20] (central) to (v2);
\end{tikzpicture}
}
\]
\end{minipage} \hspace*{.7cm}
\begin{minipage}{0.75\textwidth}
Let $Q$ be a quiver obtained by any orientation of the graph shown. 
Order the vertices $(x, 1, 2)$ and fix the dimension vector $\alpha = (2, 3, 3)$ and stability parameter $\theta = (3, -1, -1)$. Then $\alpha$ decomposes in two different ways: \vspace{-.2cm}
\[
\alpha = (1, 3, 0) + (1, 0, 3) = (1, 2, 1) + (1, 1, 2). 
\]
\end{minipage}
\vspace*{-.2cm}

Let $\Phi = \{e_x + 3e_1, e_x + 2e_1+e_2, e_x+e_1+2e_2, e_x+3e_2\}$. Consider the function 
\[
s: \Phi \to \{ +, - \} \hspace{1cm} s(e_x + 3e_1) = + = s(e_x+3e_2), \ \ s(e_x + 2e_1+e_2) = - = s(e_x+e_1+2e_2).
\]
This defines an open subset $U_s$ of representations of $\overline{Q}$ in $\mu^{-1}(0)$ such that, for $s(\beta)=+$, there is no subrepresentation of dimension $\beta$. The quotient  $U_s /\!/ G_{\alpha}$ is a proper, nonprojective crepant resolution of $\cM_{\theta}(Q, \alpha)$.

\subsubsection{Four Vertex Example}
\begin{minipage}{0.75\textwidth}
  \setlength{\parindent}{15pt}
Let $Q$ be a quiver obtained by any orientation of the graph shown. Order the vertices $(x, 1, 2, y)$. Fix the dimension vector $\alpha = (1, 2, 2, 1)$ and stability parameter $\theta = (2, -1, -1, 2)$. 

\indent Let $\Phi := \{ e_x + 2e_1, e_x + e_1 + e_2, e_x + 2 e_2 \}$ and define
\[
s: \Phi \to \{ \pm \} \hspace{1cm} s(e_x + 2e_1) = + = s(e_x + 2e_2), \ \ s(e_x + e_1 + e_2) = -.
\]
$U_s /\!/ G_{\alpha}$ defines a nonprojective crepant resolution of $\cM_{\theta}(Q, \alpha)$.
\end{minipage} \hspace{0.2cm}
\begin{minipage}{0.2\textwidth}
\[
\begin{tikzpicture}
    \node[circle, fill=blue!20, draw = blue, thick, inner sep = 0.5ex] (central) at (0, -1.5) {$x$};
    \node[circle, fill=black!20, draw = black, thick, inner sep = 0.5ex] (v1) at (-1.5, 0) {1};
    \node[circle, fill=black!20, draw = black, thick, inner sep = 0.5ex] (v2) at (1.5, 0) {2};
    \node[circle, fill=red!20, draw = red, thick, inner sep = 0.5ex] (infinity) at (0, 1.5) {$y$};

    \draw [-, shorten >=1mm, shorten <=1mm, bend right = 15] (central) to (v1);
    \draw [-, shorten >=1mm, shorten <=1mm, bend left = 15] (central) to (v2);
    \draw [-, shorten >=1mm, shorten <=1mm, bend right = 15] (infinity) to (v1);
    \draw [-, shorten >=1mm, shorten <=1mm, bend left = 15] (infinity) to (v2);

    \draw [-, shorten >=1mm, shorten <=1mm, bend left=15] (central) to (v1);
    \draw [-, shorten >=1mm, shorten <=1mm, bend right=15] (central) to (v2);
    \draw [-, shorten >=1mm, shorten <=1mm, bend left=15] (infinity) to (v1);
    \draw [-, shorten >=1mm, shorten <=1mm, bend right=15] (infinity) to (v2);
\end{tikzpicture}
\]
\end{minipage}

    \subsection{General construction}\label{ss:stars}
    
    Consider a quiver with vertices $Q_0=\{x,1, 2, \ldots,n,y\}$ for $n \geq 2$ and with arrow set $Q_1$. Take a  dimension vector $\alpha \in \bN^{Q_0}$ with $\alpha_x=\alpha_y=1$.
   Let us assume that  $\alpha \in \Sigma_0$ for $\Sigma_0$ the set of dimension vectors of simple representations of $\overline{Q}$ lying in $\mu^{-1}(0)$; this will be satisfied if for instance we have $\alpha_i$ arrows from each of $x$ and $y$ to $i$ for every $i$. As is well known, this assumption guarantees that $\cM_\theta(Q,\alpha) \to \cM_0(Q,\alpha)$ is birational for all $\theta \in \bZ^{Q_0}$ with $\theta \cdot \alpha = 0$, because all simple representations are $\theta$-stable, and the image of the locus of these representations forms an open dense subset of $\cM_\theta(Q,\alpha)$ which maps isomorphically to the open dense simple locus in $\cM_0(Q,\alpha)$.
    
    Consider a stability condition $\theta \in \bZ^{Q_0}$ with $\theta_x,\theta_y > 0$ and $\theta_i < 0$ for all $1 \leq i \leq n$.
    Then, all of the non-open symplectic leaves are closed, and are described by a decomposition $\alpha=\beta + (\alpha-\beta)$ where $e_x \leq \beta \leq \alpha$ and $\beta \cdot \theta = 0$. Call the corresponding leaf $S_\beta$.
    
    The local singularity at each closed symplectic leaf $S_\beta$ is described by a quiver with two vertices, 
    $m_\beta := -(\beta,\alpha-\beta) \geq 2$ arrows between them, and possibly some loops at the vertices. By our assumption that $\alpha \in \Sigma_0$, it follows that $m_\beta \geq 2$. If $m_\beta > 2$,  there are exactly two crepant resolutions of this local singularity, whereas if $m_\beta = 2$, then this singularity has a unique resolution as it is a du Val singularity.
    
    Any choices of local resolutions as above can be assembled to a locally projective crepant resolution of $\cM_\theta(Q,\alpha)$,  constructible as a quotient $U_s/\!/G_\alpha$ for $s: \Phi \to \{+,-\}$, where now 
    \[
    \Phi =\{\beta \in R_+: e_x \leq \beta \leq \alpha, \ \beta \cdot \theta = 0, \ m_\beta > 2\}.
    \]
    Further, by construction, they are pairwise nonisomorphic.  
\begin{thm}
    If $s: \Phi \to \{+,-\}$ has the property that, for some $k \geq 2$ and some $\beta_1, \ldots, \beta_k, \gamma_1, \ldots, \gamma_k \in \Phi$, we have
    \[
    \sum_{i=1}^k \beta_i = \sum_{i=1}^k \gamma_i,\]
    with $s(\beta_i) = +, \ s(\gamma_i) = -$,
    then the (proper) crepant resolution $U_s/\!/G_\alpha \to \cM_0(Q,\alpha)$ is nonprojective.
\end{thm}
The proof is essentially the same as for Theorem \ref{thm:four-point}, and is omitted.\\
Moreover, we can see that any nonprojective resolutions  $U_s/\!/G_\alpha \to \cM_0(Q,\alpha)$ can be extended to enlargements of the quiver:
  \begin{cor}\label{c:extend-resolutions}
    Let $(Q, \alpha, \theta)$ be as above. 
    Let $Q'$ be a quiver containing $Q$ and $\alpha'$ be a dimension vector with 
    $\alpha'_x=\alpha'_y=1$ and $\alpha'_i \geq \alpha_i$ for all $1 \leq i \leq n$. Assume that $\alpha' \in \Sigma_0$ for $Q'$. 
    Let $\theta' \in \bZ^{Q_0'}$ be a stability parameter with $\theta'_x=\theta_x$,  $\theta_i=\theta'_i$ for $1 \leq i \leq n$, and with all values $\theta'_i$ negative except for $\theta'_x,\theta'_y$.
    Then $\Phi' \supseteq \Phi$. If $s: \Phi \to \{+,-\}$ is such that $U_s/\!/G_\alpha \to \cM_0(Q,\alpha)$ is nonprojective, then for any extension of $s$ to $s': \Phi' \to \{ +, - \}$, 
    $U'_{s'}/\!/G_{\alpha'} \to \cM_{0}(Q', \alpha')$ is also nonprojective.
\end{cor}

\begin{proof}
    Note that a root for $Q$ is still a root for $Q'$. So it follows immediately that $\Phi' \supseteq \Phi$, if we think of elements of $\bN^{Q}$ as elements of $\bN^{Q'}$ by extension by zero to vertices in $Q'_0\setminus Q_0$.  Suppose that $s': \Phi' \to \{+,-\}$ extends $s$. 
    
    We again use that any projective resolution of $\cM_0(Q',\alpha')$ is given by variation of GIT quotient \cite{BCS23}, and the ones which factor through $\cM_{\theta'}(Q',\alpha')$ are given as $\cM_{\chi'}(Q',\alpha') \to \cM_{\theta'}(Q',\alpha')$ for $\chi'$ in a GIT chamber whose boundary contains $\theta'$. Then we can construct $t': \Phi' \to \{+,-\}$ such that $U_{t'} = \mu_{\alpha'}^{-1}(0)^{\chi'\text{-ss}}$.  

    Thus, if $U_{s'}/\!/G_{\alpha'} \to \cM_0(Q',\alpha')$ is projective, we have $U_{s'}/\!/G_{\alpha'} \cong U_{t'}/\!/G_{\alpha'}$ for some $t'$ with $U_{t'} = \mu_{\alpha'}^{-1}(0)^{\chi'\text{-ss}}$.  However, all of the different resolutions $U_{s'}/\!/G_\alpha$ for varying $s'$ are nonisomorphic over $\cM_{\theta'}(Q',\alpha')$, hence over $\cM_0(Q',\alpha')$.  So $s'=t'$, and $U_{s'}=\mu_{\alpha'}^{-1}(0)^{\chi'\text{-ss}}$.

    
    But then if we let $\chi \in \bZ^{Q_0}$ be the stability condition with $\chi_x=\chi'_x$, $\chi_i=\chi'_i$ for $1 \leq i \leq n$, and $\chi \cdot \theta=0$, then we obtain that $U_s=\mu_{\alpha}^{-1}(0)^\chi$. So $U_s/\!/G_\alpha \to \cM_0(Q,\alpha)$ is also projective. 
\end{proof}


    \begin{exam} \label{exam:nonhyptertoric}
        Let $Q, \alpha, \theta$ be as in Section \ref{ss:four-pointed}, and let $Q'$ be the quiver with $Q'_0=Q_0$ and with two arrows from each of $x$ and $y$ to $1,2,3$, and $4$. Let the dimension vector be $\alpha'=(1,2,2,2,2,1)$. To apply Corollary \ref{c:extend-resolutions}, set  $\theta'=(2,-1,-1,-1,-1,6)$. The minimal strata now consist of decompositions of the form $\alpha'=\beta+(\alpha'-\beta)$ with $\beta = e_x + e_i + e_j$ for $1 \leq i \leq j$. There are now ten (not six) minimal strata. As explained in Section \ref{ss:local}, each singularity has two local choices of projective crepant resolution. The 18 choices of crepant resolutions from before along the six strata corresponding to $i \neq j$ still are incompatible with all of the global projective crepant resolutions (for any choice of crepant resolutions along the remaining four strata).
    \end{exam}

    Now, let us restrict to the case where $\alpha_i=1$ and $\theta_i=-1$ for all $1 \leq i \leq n$, and $Q_1$ contains at least one arrow from $x$ to each of $1,\ldots,n$. Assume also that $n \geq 3$ and $m:=\theta_x \in \{1,\ldots, n-1\}$. In this case $\Phi=\{\beta_I: I \subseteq \{1,\ldots,n\} \mid |I|=m\}$, with $\beta_I := e_x + \sum_{i \in I} e_i$. As a result $|\Phi| = {n \choose m}$ and there are precisely $2^{{n \choose m}}$ locally projective crepant resolutions of $\cM_\theta(Q,\alpha)$. We obtain:
    \begin{cor} \label{c:general_nonprojective} 
    Under the preceding assumptions, if $s: \Phi \to \{+,-\}$ has the property that, for some $k \geq 2$ and some subsets $I_1,\ldots, I_k$ and $J_1,\ldots,J_k$ of $\{1,\ldots,n\}$ of size $m$ such that
    \[\sum_{i=1}^k \beta_{I_i} = \sum_{i=1}^k \beta_{J_i}, \]
    i.e., the unions with multiplicity of the $I_i$ and $J_i$ are the same, we have
    \[
    s(\beta_{I_i}) = +, \quad s(\beta_{J_i})= -, \quad \forall i=1,\ldots, k,
    \]
    then the (proper) crepant resolution $U_s/\!/G_\alpha \to \cM_0(Q,\alpha)$ is nonprojective.
    \end{cor}
    
We next randomly generate a function $s: \Phi \to \{ +, - \}$, and consider the associated resolution of the form $U_s/\!/G_\alpha \to \cM_0(Q,\alpha)$. 
    
    \begin{cor} \label{cor:probability}
       Continuing with the assumptions of the preceding corollary, a randomly generated resolution of the form $U_s/\!/G_\alpha \to \cM_0(Q,\alpha)$ is nonprojective with probability approaching $1$ as $n$ and $n-m$ tend to infinity.
    \end{cor} 

    \begin{proof}
        We will analyze the condition of Corollary \ref{c:general_nonprojective} for $k=2$ and partitions $I_1 \sqcup I_2 = \{ 1, \ldots, 2m \} \subset \{1, \ldots, n \}$ with $|I_1| = m = |I_2|$. There are $M :={2m \choose m}/2$ such partitions. Fixing such a partition $I_1, I_2$, the probability that $s$ does not satisfy $s(\beta_{I_1})=+=s(\beta_{I_2})$ is $3/4$. So the probability that $s$ does not satisfy the condition on all partitions is $(3/4)^M$. Similarly, the probability that $s$ does not satisfy $s(\beta_{J_1})=-=s(\beta_{J_2})$ for some choice of partition $J_1, J_2$ is $(3/4)^M$. Therefore, with probability approaching $1$, the hypotheses of the theorem will be satisfied.
    \end{proof}    
    Based on the evidence given, we expect that, for sufficiently large quivers, the vast majority of crepant resolutions are nonprojective (as shown for the hyperpolygon case in \cite{Hubbard}). 

\section{Good quotient constructions}\label{s:open-quotient}

In this section, as in \cite{AP16,Hubbard}, we construct resolutions of singularities by more general quotients than the usual ones provided by GIT. Given a singularity $V/\!/G$, we consider certain $G$-stable open subsets $U \subseteq V$, which yields the map $\pi_U: U/\!/G \to V/\!/G$.  This construction subsumes the examples in the previous section. In those cases $\pi_U$ was proper, being a composition of locally projective (hence proper) morphisms. Here, the main challenge is determining for which $G$-stable open subsets $U$ we have that $\pi_U$ is proper.  We will prove (sufficient) criteria on $U$ for $\pi_U$ to be proper. This construction yields (proper) crepant resolutions which cannot be obtained by iteratively taking locally projective resolutions.

\subsection{Some $G$-stable open subsets of representations}

Let us explicitly describe the set $U_s$ from the previous section for $Q_0=\{x,1,2,\ldots,n,y\}$, with  $Q_1$ any set of arrows including at least one of the form $x \to i$ and $y \to j$ for every $i,j \in \{1,\ldots,n\}$. Let dimension vector $\alpha=(1,\ldots,1)$ and let $\theta=(m,-1,\ldots,-1,n)$ for $2 \leq m \leq n-2$. We consider locally projective resolutions of $\cM_\theta(Q,\alpha)$, which up to isomorphism are given by maps $s: \Phi \to \{+,-\}$ with $\Phi=\{e_x+e_{i_1}+\cdots+e_{i_m}: 1 \leq i_1 < \cdots < i_m \leq n\}$. This defines a $G$-stable open subset $U_s \subset \mu_\alpha^{-1}(0)^{\theta-ss}$. \\

For concreteness, consider the case $n=4$, $m=2$. Here
\[
U_s \subset \mu_{\alpha}^{-1}(0)^{\theta-\text{ss}} \subset \Rep_{\alpha}(\overline{Q}) = \bA^{16}.
\]
Fix a representation $\rho =(V_i, \ T_a: i \in Q_0, \ a \in \overline{Q}_1) \in \Rep_{\alpha}(\overline{Q})$. Then King's stability criterion says $\rho$ is $(2,-1,-1,-1,-1,2)$-semistable if it has no subrepresentation $\rho'$ with $\dim(\rho') \cdot \theta >0$. Therefore every subrepresentation containing $V_x$ or $V_y$ must also contain $V_i, V_j$ for some distinct $i,j \in \{1,2,3,4\}$, and any subrepresentation containing both  $V_x$ and $V_y$ must be the entire representation, $\rho$.  

Next, given $s: \Phi \to \{+,-\}$, we additionally require that if $s(\beta)=+$ then some arrow defines a nonzero linear map from $V_{\beta}$ to $V_{\alpha-\beta}$. Explicitly this means that, for some $i, j \in Q_0$ in the support of $\beta, \alpha-\beta$, respectively, there is an arrow $a:i \to j$ in $\overline{Q}_1$ for which $T_a \neq 0$. By our previous assumptions, for $\beta=e_x+e_i+e_j$ and $\alpha-\beta=e_y+e_k+e_\ell$, it is equivalent to ask that, if $s(\beta)=+$, there is a nonzero \emph{path} from $x$ to $k$ or $\ell$, and if $s(\beta)=-$, there is a nonzero \emph{path} from $y$ to $i$ or $j$ (the path can be a composition of arrows).

We can restate all of this simply as follows:  given a quiver $Q$ with dimension vector $\alpha=(1,\ldots,1)$, the open set $U_s$ is the union of the ${4 \choose 2}$ loci
\[
U_s = \bigcup_{\substack{I \subseteq  \{ 1, 2, 3, 4\}, \\ |I| =2 }} U_{s, I},
\]
where $U_{s, I} \subseteq \mu_\alpha^{-1}(0)$ is the subset of representations with:
\begin{itemize}
\item nonzero paths from $x$ to every vertex of $I$, 
\item nonzero paths from $y$ to every vertex of $I^c$, and 
\item $\begin{cases} \text{a nonzero path from }x\text{ to some vertex of }I^c & \text{ if }s(\beta_I)=+, \\ 
 \text{a nonzero path from }y\text{ to some vertex of }I & \text{ if } s(\beta_I)=-.
    \end{cases}$ \\
\end{itemize}

Similarly, for general $m,n \geq 2$, we obtain a description of $U_s$ as a union
\[
U_s = \bigcup_{\substack{I \subseteq \{1,\ldots,n\},  \\ |I| = m}} U_{s, I}
\]
where $U_{s, I}$ is defined as above. 
\subsection{A crepant resolution which does not factor through a projective partial resolution}

Armed with the preceding description, we generalize the construction by replacing the single condition $|I|=m$ by a more flexible condition on the subsets $I$; at the same time we replace the function $s$ and the complement $I^c$  by a more symmetric condition for subsets $I$ and $J$ to which $x$ and $y$ will have nonzero paths, respectively.  

We first present this in the setting of $Q_0=\{x,1,2,3,4,5,y\}$, with $Q_1$ consisting of one arrow from each of $x,y$ to each of $1,2,3,4,5$ (so ten arrows total), and $\alpha=(1,1,1,1,1,1,1)$. This is equivalent to the framed five-pointed
star with framing of one at each endpoint.

Let $\mathcal{I}$ be the set of subsets $I \subseteq \{1,2,3,4,5\}$ such that
\begin{equation}\label{e:fivepoint-icond}
I \supseteq \{1,2\} \quad \text{\emph{or}} \quad I \supseteq \{3, 4,5\}.
\end{equation}
Let $\mathcal{J}$ be the set of subsets $J \subseteq \{1,2,3,4,5\}$ be such that
\[
J \cap \{1,2\} \neq \emptyset \quad \text{\emph{and}} \quad J \cap \{3,4,5\} \neq \emptyset.
\]
Then, we consider the open subset $U \subseteq \mu_{\alpha}^{-1}(0)$ of representations of $\overline{Q}$ of dimension $\alpha=(1,1,1,1,1,1,1)$ such that:
\begin{itemize}
\item For some $I \in \mathcal{I}$, there  is a nonzero path from $x$ to each $i \in I$;
\item For some $J \in \mathcal{J}$, there is a nonzero path from $y$ to each $j \in J$;
\item For every $i \in \{1,2,3,4,5\}$, there is a nonzero path from either $x$ or $y$ to $i$.
\end{itemize}

\begin{thm}\label{t:five-point-git}
\begin{itemize}
    \item[(a)] There is a geometric quotient  $X:=U/\!/G_\alpha$ equipped with a canonical morphism $\pi_U: X \to \cM_0(Q,\alpha)$ which is a (proper) crepant resolution of singularities.
    \item[(b)] There does not exist a nontrivial projective crepant partial resolution $Y \to \cM_0(Q,\alpha)$ such that $\pi_U$ factors as $X \to Y \to \cM_0(Q,\alpha)$.
\end{itemize}
\end{thm}
We prove Theorem \ref{t:five-point-git}(a) later, as a consequence of a  more general result (Theorem \ref{t:stars-git}). The main step establishes that $\pi_U$ is \emph{proper}, which follows from a careful application of the valuative criterion of properness: given a discrete valuation ring $R$ with 
fraction field $K$, the proof boils down to showing that every $G_\alpha(K)$ orbit in $U(K)$
which lies over $\cM_0(Q,\alpha)(R)$ contains a point of $U(R)$.


The proof of Theorem \ref{t:five-point-git}(b) is based on the following lemma:
\begin{lem} The only $\theta \in \bZ^{Q_0} \cap \alpha^\perp$ for which $U \subseteq \mu^{-1}(0)^{\theta-ss}$ is $\theta=0$.
\end{lem}
\begin{proof}
By definition, $U$ contains the following representations for all $I \in \mathcal{I}, J \in \mathcal{J}$:
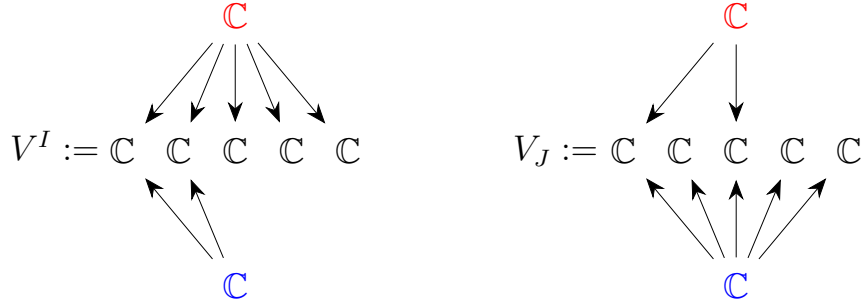
\[
\scalebox{0.9}{
 \begin{tikzpicture}
    \node (label) at (-2/1.2-1, 0.08) {\scalebox{1.25}{$V^I:=$}};
    \node (central) at (0,-2) {\textcolor{blue}{\scalebox{1.25}{$\mathbb{C}$}}};
    \node (v1) at (-2/1.2, 0) {\scalebox{1.25}{$\mathbb{C}$}};
    \node (v2) at (-1/1.2, 0) {\scalebox{1.25}{$\mathbb{C}$}};
    \node (v3) at (0, 0) {\scalebox{1.25}{$\mathbb{C}$}};
    \node(v4) at (1/1.2, 0) {\scalebox{1.25}{$\mathbb{C}$}};
    \node (v5) at (2/1.2, 0) {\scalebox{1.25}{$\mathbb{C}$}};
    \node(infinity) at (0, 2) {\textcolor{red}{\scalebox{1.25}{$\mathbb{C}$}}};
    \draw [-{Stealth[length=3mm]}, shorten >=1mm, shorten <=1mm] (central) -- (v1);
    \draw [-{Stealth[length=3mm]}, shorten >=1mm, shorten <=1mm] (central) -- (v2);
    \draw [-{Stealth[length=3mm]}, shorten >=1mm, shorten <=1mm] (infinity) -- (v1);
    \draw [-{Stealth[length=3mm]}, shorten >=1mm, shorten <=1mm] (infinity) -- (v2);
    \draw [-{Stealth[length=3mm]}, shorten >=1mm, shorten <=1mm] (infinity) -- (v3);
    \draw [-{Stealth[length=3mm]}, shorten >=1mm, shorten <=1mm] (infinity) -- (v4);
    \draw [-{Stealth[length=3mm]}, shorten >=1mm, shorten <=1mm] (infinity) -- (v5);
\end{tikzpicture} \hspace{1.7cm}
 \begin{tikzpicture}
    \node (label) at (-2/1.2-1, 0) {\scalebox{1.25}{$V_J:=$}};
    \node (central) at (0,-2) {\textcolor{blue}{\scalebox{1.25}{$\mathbb{C}$}}};
    \node (v1) at (-2/1.2, 0) {\scalebox{1.25}{$\mathbb{C}$}};
    \node (v2) at (-1/1.2, 0) {\scalebox{1.25}{$\mathbb{C}$}};
    \node (v3) at (0, 0) {\scalebox{1.25}{$\mathbb{C}$}};
    \node(v4) at (1/1.2, 0) {\scalebox{1.25}{$\mathbb{C}$}};
    \node (v5) at (2/1.2, 0) {\scalebox{1.25}{$\mathbb{C}$}};
    \node(infinity) at (0, 2) {\textcolor{red}{\scalebox{1.25}{$\mathbb{C}$}}};
     \draw [-{Stealth[length=3mm]}, shorten >=1mm, shorten <=1mm] (central) -- (v1);
    \draw [-{Stealth[length=3mm]}, shorten >=1mm, shorten <=1mm] (central) -- (v2);
    \draw [-{Stealth[length=3mm]}, shorten >=1mm, shorten <=1mm] (central) -- (v3);
    \draw [-{Stealth[length=3mm]}, shorten >=1mm, shorten <=1mm] (central) -- (v4);
     \draw [-{Stealth[length=3mm]}, shorten >=1mm, shorten <=1mm] (central) -- (v5);
    \draw [-{Stealth[length=3mm]}, shorten >=1mm, shorten <=1mm] (infinity) -- (v1);
    \draw [-{Stealth[length=3mm]}, shorten >=1mm, shorten <=1mm] (infinity) -- (v3);
\end{tikzpicture}
}
\]
where each depicted arrow is the identity and each omitted arrow is zero. For $V^I$ (resp. $V_J$) the only nonzero maps from \textcolor{blue}{$x$} (resp. \textcolor{red}{$y$}) land in $I$ (resp. $J$). 

Notice that for each $i \in \{1, 2, \dots, 5 \}$, $V^I$ has a subrepresentation, $S_i$, supported on $\{i \}$. Further $V^I$ has a subrepresentation, $W^I$, supported on $I \cup \{ x \}$ and $V_J$ has a subrepresentation, $W_J$, supported on $J \cup \{ y \}$.  

Assume that $U \subseteq \mu^{-1}(0)^{\theta-\text{ss}}$, meaning each representation in $U$ is $\theta$-semistable. Then each subrepresentation $W$ has dimension vector satisfying $\dim(W) \cdot \theta \leq 0$. This imposes (at least) the following three classes of inequalities on the coordinates of $\theta = (\theta_x, \theta_1, \theta_2, \dots, \theta_5, \theta_y)$ (writing $\theta_I := \sum_{i \in I} \theta_i$):
\begin{itemize}
\item $S_i$ is a subrepresentation of the $\theta$-semistable $V^I$, so  
$\theta \cdot \dim(S_i) = \theta_i \leq 0$;
\item $W^I$ is a subrepresentation of the $\theta$-semistable $V^I$, so $\theta \cdot \dim(W^I) = \theta_x + \theta_I \leq 0$;
\item $W_J$ is a subrepresentation of the $\theta$-semistable $V_J$, so $\theta \cdot \dim(W_J) = \theta_y + \theta_J \leq 0$.
\end{itemize}
Using the equality $0 = \alpha \cdot \theta = \theta_x + \theta_1 + \cdots + \theta_5 + \theta_y$ we can transform each above inequality as follows:
\[
\theta_x + \theta_I \leq 0 \ \text{ and } \  \alpha \cdot \theta = 0  \ \implies \ \theta_y + \theta_{I^c} \geq 0.
\]
In particular, we get a string of inequalities
\[
 \theta_y + \theta_{I^c} \geq 0 \geq  \theta_y + \theta_J   \implies \theta_{I^c} - \theta_{J} \geq 0.
\]
Taking $J = (I^c \cup \{ i \}) \backslash \{ j \}$ gives $\theta_j - \theta_i \geq 0$. By taking $I = \{ 1, 2 \}$ or $\{ 3, 4, 5 \}$ and varying $i$ and $j$ we get $\theta_i = \theta_j$ for all $i, j$. Now setting $I = \{ 1, 2\}$ and $J = \{1, 3 \}$ gives $0 = \theta_4+\theta_5 -\theta_1 = \theta_4$. Hence all $\theta_i = 0$ which implies $\theta_x = 0 = \theta_y$. We conclude that $\theta = 0$. 
\end{proof}

To deduce from the lemma that $U/\!/G_\alpha$ does not factor through a nontrivial projective partial resolution, we will show that its only globally generated line bundle is trivial. In turn we need to understand its Picard group of isomorphism classes of line bundles. For this purpose we will compare it with $\cM_\theta(Q,\alpha)$.  For the latter we have the following result of McGerty--Nevins:
\begin{thm}[\cite{MN_Kirwansurj} Theorem 1.2, \cite{BCS23} Theorem 4.2]   \label{t:MN_Kirwansurj}
If $\alpha \in \Sigma_0$ (there is a simple representation in $\mu_\alpha^{-1}(0)$) and $\alpha_i=1$ for some $i \in Q_0$, then for generic $\theta \in \bZ^{Q_0} \cap \alpha^\perp$, descent furnishes a surjection
$L_\theta: PG_{\alpha}^\vee \to \Pic(\cM_\theta(Q,\alpha))$.
\end{thm}
\begin{rem}
In the example at hand, $\alpha=(1,\ldots, 1)$, the smooth variety $\cM_\theta(Q,\alpha)$ is toric hyperk\"ahler. Therefore, an alternative argument is possible using tilting bundles; see the first paragraph in the proof of \cite[Theorem 5.1]{BCS23}.
\end{rem}
\par\noindent
\begin{minipage}{0.75\textwidth} 
\noindent 
Let $S$ be the representation of $\overline{Q}$ shown on the right, where each depicted arrow is the identity and each omitted arrow is zero. Notice that $S$ is simple and lies in $\mu^{-1}(0)$, so $\alpha \in \Sigma_0$.

\hspace*{2em} Since further $\alpha_i=1$ for all $i \in Q_0$, Theorem \ref{t:MN_Kirwansurj} applies to the pair $(Q,\alpha)$ at hand.  To deduce the corresponding result for $U/\!/G_\alpha$, we compare $U$ with the $\theta$-semistable locus for the ``egalitarian'' $\theta$ giving $\theta_i$ the same values for $1 \leq i \leq 5$, and $\theta_x=\theta_y$:
\end{minipage} \hspace{.25cm}
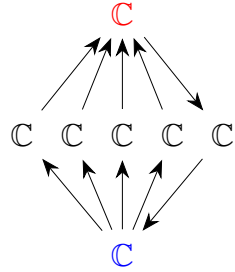
\begin{minipage}{0.3\textwidth}
\scalebox{0.8}{
 \begin{tikzpicture}
    \node (central) at (0,-2) {\textcolor{blue}{\scalebox{1.25}{$\mathbb{C}$}}};
    \node (v1) at (-2/1.2, 0) {\scalebox{1.25}{$\mathbb{C}$}};
    \node (v2) at (-1/1.2, 0) {\scalebox{1.25}{$\mathbb{C}$}};
    \node (v3) at (0, 0) {\scalebox{1.25}{$\mathbb{C}$}};
    \node(v4) at (1/1.2, 0) {\scalebox{1.25}{$\mathbb{C}$}};
    \node (v5) at (2/1.2, 0) {\scalebox{1.25}{$\mathbb{C}$}};
    \node(infinity) at (0, 2) {\textcolor{red}{\scalebox{1.25}{$\mathbb{C}$}}};
    \draw [-{Stealth[length=3mm]}, shorten >=1mm, shorten <=1mm] (central) -- (v1);
    \draw [-{Stealth[length=3mm]}, shorten >=1mm, shorten <=1mm] (central) -- (v2);
    \draw [-{Stealth[length=3mm]}, shorten >=1mm, shorten <=1mm] (central) -- (v3);
    \draw [-{Stealth[length=3mm]}, shorten >=1mm, shorten <=1mm] (central) -- (v4);
     \draw [-{Stealth[length=3mm]}, shorten >=1mm, shorten <=1mm] (v5) -- (central);
    \draw [-{Stealth[length=3mm]}, shorten >=1mm, shorten <=1mm] (v1) -- (-0.3, 1.7);
    \draw [-{Stealth[length=3mm]}, shorten >=1mm, shorten <=1mm] (v2) -- (-0.15, 1.7);
    \draw [-{Stealth[length=3mm]}, shorten >=1mm, shorten <=1mm] (v3) -- (0, 1.7);
    \draw [-{Stealth[length=3mm]}, shorten >=1mm, shorten <=1mm] (v4) -- (0.15, 1.7);
    \draw [-{Stealth[length=3mm]}, shorten >=1mm, shorten <=1mm] (0.3, 1.7) -- (v5);
\end{tikzpicture}
}
\end{minipage}

\begin{lem}\label{l:U-theta-three}
Let $\theta = (5,-2,-2,-2,-2,-2,5)$.  Then the $\theta$-unstable locus of $U$ has codimension at least three in $U$.
\end{lem}
\begin{proof}
A representation is $\theta$-unstable if and only if one of the following holds:
\begin{itemize}
\item Uns($i$), $i \in \{1,2,3,4,5\}$: Each arrow to vertex $i$ is zero;
\item Uns($x$) (respectively Uns($y$)): All paths from $x$ (resp. $y$) to $i$ are zero for at least three $i \in \{ 1, 2, \dots, 5 \}$;
\end{itemize}
For Uns($i$) the subrepresentation $W_i$ supported on the complement of $i$ is $\theta$-destabilizing. For Uns($x$) the subrepresentation $W_x$ supported on $c$ and all vertices accepting a non-zero path from $x$ is $\theta$-destabilizing. (The same holds replacing $x$ with $y$ everywhere in the previous sentence.) Conversely if we are given a $\theta$-destabilizing subrepresentation, its support must contain $x$ or $y$ and, if only one of these, then at most two of $1,2,3,4,5$. This happens if and only if one of the above conditions holds, taking into account the specific quiver $Q$.

By definition a representation in $U$ does not satisfy condition Uns($i$).
If $\rho \in U$ satisfies condition Uns($x$), then at most two of the five two-cycles in $\overline{Q}$ at $x$ are nonzero. The condition $\rho \in \mu^{-1}(0)$ implies these cycles are negatives of each other. This is a codimension-three condition in $U$. 
\end{proof}

\begin{rem}A similar argument shows also that $\mu^{-1}(0)^{\theta-\text{ss}} \setminus U$ has codimension at least three in $\mu^{-1}(0)^{\theta-\text{ss}}$; from this one can quickly deduce that $\Pic(\cM_\theta(Q,\alpha)) \cong \Pic(U/\!/G_\alpha)$. But we will apply an argument that uses a bit less.
\end{rem}

A $PG_\alpha$-equivariant structure (``$PG_\alpha$-linearization'') on the trivial line bundle $U \times \bC \rightarrow U$ is given by a choice of character $\chi \in PG_\alpha^{\vee}$. This defines a linearization map
\[
L_U: PG^\vee \to \Pic(U/\!/PG_\alpha)
\]
sending each character to its $PG_\alpha$-equivariant trivial line bundle. 

\begin{lem}\label{l:UG-lin-surj} The linearization map $L_U$ is surjective.
\end{lem}
\begin{proof}Consider the diagram with $i, j$ open embeddings and vertical quotient maps: 
\[
\begin{tikzcd}
 U  \arrow[d] & \arrow[l, "i" ']   U^{\theta-ss} \arrow[r,"j"] \arrow[d] & \mu^{-1}(0)^{\theta-ss} \arrow[d] \\
 U/\!/PG_\alpha & \arrow[l,"\overline{i}" '] U^{\theta-ss}/\!/PG_\alpha \arrow[r,"\overline{j}"] & \cM_\theta(Q,\alpha)
\end{tikzcd}
\]
It follows from Lemma \ref{l:U-theta-three} that the restriction map $i^*: \Pic_G(U) \to \Pic_G(U^{\theta\text{-ss}})$, and hence by descent, $\overline{i}^*: \Pic(U/\!/PG_\alpha) \to \Pic(U^{\theta-\text{ss}}/\!/PG_\alpha)$, is an isomorphism.  On the other hand, since $\overline{j}$ is an open embedding of smooth (hence locally factorial) varieties, $\overline{j}^*: \Pic(\cM_\theta(Q,\alpha)) \to \Pic(U^{\theta-\text{ss}} /\!/ PG_{\alpha})$ is surjective.  So it follows that all line bundles on $U/\!/PG_\alpha$ are isomorphic to those of the form $L_U(\chi)$ for some $\chi \in PG^\vee$.
\end{proof}
\begin{proof}[Proof of Theorem \ref{t:five-point-git}(b)]
Suppose that there were a morphism $\pi: U/\!/G \to Y$ for some nontrivial projective partial resolution $Y \to \cM_0(Q,\alpha)$, such that the composition is the canonical morphism $U/\!/G \to \cM_0(Q,\alpha)$.  Then, letting $E$ be an ample line bundle on $Y$, $\pi^* E$ would be a stably globally generated, nontrivial line bundle on $U/\!/G$.  By Lemma \ref{l:UG-lin-surj}, it must be of the form $L_U(\chi)$ for some nontrivial $\chi \in PG^\vee$.  
Since $U \not \subseteq \mu^{-1}(0)^{\chi-ss}$, let $u \in U \setminus \mu^{-1}(0)^{\chi\text{-ss}}$ be an unstable point.  Then all of $\mathcal{O}(U)^{\chi}$ vanishes at $u$.  Letting $\overline{u} \in U/\!/PG_\alpha$ be the image, we see that all sections of $L_\chi^m$ vanish at $\overline{u}$ for all $m \geq 1$. So $L_U(\chi)$ is not stably globally generated. This is a contradiction.
\end{proof}

\subsection{Proper crepant resolutions for general quivers}  

Next, let $Q_0=\{x,1,\ldots,n,y\}$  and $Q_1$ be any set of arrows that includes at least one arrow $x \to i$ and $y \to i$ for each $i \in \{1,\ldots,n\}$, and consider dimension vector $\alpha=(1,\ldots,1)$.  (Alternatively, one can consider a framed quiver on $Q_0=\{x,1,\ldots,n\}$, where  $Q_1$ contains a star with arrows from central vertex $x$ to the other vertices. The dimension vector is $(1,\ldots,1)$, and there is a 1-dimensional framing at vertices $\{1,\ldots,n\}$.)

The construction is based on an involution on the collection of upward-closed subsets of $\{1,\ldots,n\}$, which may be of independent interest:
\begin{defn}
A subset $\mathcal{I} \subseteq \mathcal{P}(\{1,\ldots,n\})$ is \emph{upward-closed} if for every $J \subseteq \{1,\ldots,n\}$ such that $J \supseteq I$ for some $I \in \mathcal{I}$, then $J \in \mathcal{I}$ as well.
\end{defn}
To represent $\mathcal{I}$, it is convenient to use a minimal subset $\mathcal{I}_0 \subseteq \mathcal{I}$ such that every element contains one of the sets in $\mathcal{I}_0$.  For example, if $\mathcal{I}$ is the set of all subsets containing $1$, then we can take $\mathcal{I}_0 = \{\{1\}\}$.  Note that $\mathcal{I}_0$ is an \emph{antichain}, i.e., a collection of sets, none of which contains another.  In fact, $\mathcal{I}_0$ is unique: it consists precisely of the sets in $\mathcal{I}$ which do not contain another set. This produces a well-known bijection $\mathcal{I} \leftrightarrow \mathcal{I}_0$ between upward-closed collections and antichains.

Now let $\mathcal{I} \subseteq \mathcal{P}(\{1,\ldots,n\})$ be an upward-closed subset.
We now define, from $\mathcal{I}$, another set $\mathcal{J} \subseteq \mathcal{P}(\{1,\ldots,n\})$, as follows:
\begin{equation}
\mathcal{J} := \{J \subseteq \{1,\ldots,n\} \mid I \cap J \neq \emptyset,  \text{ for all } I \in \mathcal{I} \}.
\end{equation}
Notice  that $\mathcal{J}$ is upward-closed. We have another characterization of $\mathcal{J}$:
\begin{lem} We have the formula
\[
\mathcal{J} = \{I^c: I \subseteq \{1,\ldots,n\}, I \notin \mathcal{I}\}.
\]
\end{lem}
\begin{proof}
If $I \notin \mathcal{I}$, then for every $I' \in \mathcal{I}$, we have $I' \not \subseteq I$, so that $I^c \cap I' \neq \emptyset$. Conversely, if $J \cap I \neq \emptyset$ for all $I \in \mathcal{I}$, then $J^c \not \supseteq I$ for all $I \in \mathcal{I}$, or equivalently $J^c \notin \mathcal{I}$.
\end{proof}
\begin{cor}
This operation is an involution, i.e., define
\begin{equation}\label{e:ij-invol}
 \mathcal{K} := \{I \subseteq \{1,\ldots,n\} \mid \forall J \in \mathcal{J}, I \cap J \neq \emptyset\}.
\end{equation}
Then $\mathcal{K} = \mathcal{I}$.
\end{cor}
\begin{proof}
It is clear that every $I \in \mathcal{I}$ is also in $\mathcal{K}$. For the converse, suppose that $I \notin \mathcal{I}$. Then by the preceding lemma, $I^c \in \mathcal{J}$. But then $I \cap I^c = \emptyset$, so that $I \notin \mathcal{K}$.
\end{proof}

This involution fixes upward-closed subsets with a unique minimal subset. 

\begin{defn} \label{d:good_subset}
    Let $\mathcal{I} \subset \mathcal{P}( \{ 1, \dots, n \} )$ be an upward closed subset with $\mathcal{J}$ defined as above. Define $U \subseteq \mu^{-1}(0)$ to be the open subset of representations such that:
\begin{itemize}
\item[(1)] For some $I \in \mathcal{I}$, there are nonzero paths from $x$ to every vertex of $I$;
\item[(2)] For some $J \in \mathcal{J}$, there are nonzero paths from $y$ to every vertex of $J$;
\item[(3)] For every $i \in \{1,\ldots,n\}$, there is a nonzero path from either $x$ or $y$ to $i$.
\end{itemize}
\end{defn}

\begin{thm}\label{t:stars-git}
Assume that $\mathcal{I}, \mathcal{J}$ are not empty. Let $U \subseteq \mu^{-1}(0)$ be the set defined above. There is a geometric quotient $U/\!/G_{\alpha}$ such that the canonical morphism $U/\!/G_{\alpha} \to \mu^{-1}(0)/\!/G_{\alpha}$ is a (proper) crepant resolution of singularities.
\end{thm}

\begin{exam}
Taking $n=5$ and $\mathcal{I}_0=\{\{1,2\},\{3, 4,5\}\}$, then $\mathcal{J}_0$ consists of all subsets containing exactly one element of each of $\{1,2\}$ and $\{3, 4,5\}$, i.e.,   
\[\mathcal{J}_0=\{\{1,3\},\{1,4 \},\{1,5\},\{2,3\},\{2,4\},\{2,5\}\}.\]
Then, Theorem \ref{t:stars-git} specializes to Theorem \ref{t:five-point-git}(a).
\end{exam}
\begin{exam}
Taking $\mathcal{I} = \{I \subseteq \{1,\ldots,n\} \mid |I| > m \text{ or }  |I|=m \text{ and } s(\beta_I)=-\}$, then we obtain the example of Section \ref{ss:stars}.
\end{exam}
\begin{rem}
    As in Section \ref{ss:stars}, provided $\alpha \in \Sigma_0$ we can relax our assumptions on $Q$ and $\alpha$. For example, dimensions at the vertices $1,\ldots, n$ can be larger than one and the quiver does not need arrows from $x$ and $y$ to each of $1,\ldots,n$. In this case, instead of the power set of $\{1,\ldots,n\}$, consider the collection $P$ of vectors $v \in \bN^n$ with $v_i \leq \alpha_i$ for all $i$. Then define $\mathcal{I} \subseteq P$ to be upward-closed in the sense that, if $v \in \mathcal{I}$ and $v \leq v' \in P$, then also $v' \in \mathcal{I}$. The involution sends $\mathcal{I}$ to the set of vectors 
    \[
    \mathcal{J} := \{w \in P \mid (\alpha_1-w_1,\ldots,\alpha_n-w_n) \notin \mathcal{I}\}.
    \]
    One can define $U$ as in Definition \ref{d:good_subset} to obtain a proper crepant resolution $U/\!/G_{\alpha}$ of $\cM_0(Q,\alpha)$. However, different choices of $\mathcal{I}$ may yield the same set $U$: this happens when they have the same intersection with the set of vectors $v \in P$
    such that both $v+e_x$ and $\alpha-(v+e_x)$ are roots. One can remove this ambiguity by requiring that the antichain $\mathcal{I}_0$ only contain $v$ satisfying the condition that both $v+e_x$ and $\alpha-(v+e_x)$ are roots.
\end{rem}

\subsection{Proof of Theorem \ref{t:stars-git}}
Let $\pi_U$ be the map $U/\!/G_\alpha \to \mu^{-1}(0)/\!/G_\alpha$. We will show, in order, the following:
\begin{enumerate}
\item[(1)] $U/\!/G_\alpha$ is a smooth prevariety (see definition below), and a geometric quotient of $G_\alpha$; 
\item[(2)] $\pi_U$ is birational and crepant; 
\item[(3)] $U/\!/G_\alpha$ has the $A_2$ property, hence it and  $\pi_U$ are separated;
\item[(4)] $\pi_U$ is proper. 
\end{enumerate}
We will recall the relevant notions before each step.

\subsubsection{Smoothness and the geometric property of the quotient} 
To see that $U/\!/G_{\alpha}$ is geometric, we apply the following lemma.
For technical reasons we will need to slightly relax the definition of variety by not requiring separability. Following \cite[Ch. 5, Definition 1]{Mumford}, a \emph{prevariety} is a reduced scheme of finite type over $\bC$.  We extend the notion of good and geometric quotients to prevarieties, allowing non-separability. 

Let us more generally consider any reductive group $G$ acting on any affine variety $V$. We recall that a character $\chi:G \to \bC^\times$ is called \emph{generic} if every $\chi$-semistable point is $\chi$-stable, i.e., $V^{\chi-ss}=V^{\chi-s}$.
\begin{lem}\label{l:fin-stabilisers}
Let $U \subseteq V$ be any open, $G$-stable subset.
Suppose that the stabilizers of $G$ on $U$ are finite. Then there is a geometric quotient prevariety $U/\!/G$ and morphism $\pi: U/\!/G \to V/\!/G$.
\end{lem}
\begin{proof}
By \cite[Lemma 1.1]{Proudfoot-GIT-at-once}, when $G$ is a torus, then for every point $v \in V$ where $G$ acts with finite stabilizer, there exists a generic character $\chi: G \to \bC^\times$ such that $v$ is $\chi$-stable.  Now for reductive $G$, we can apply this to a maximal torus $T \leq G$.  Since every one-parameter subgroup is conjugate to a subgroup of $T$, by the Hilbert--Mumford criterion, $v \in V$ is 
$\chi$-stable if and only if it is $\chi|_T$-stable for the restricted action of $T$. 

As a result, $U$ is covered by the $\chi$-stable loci $U^{\chi-s} \subseteq U$ as $\chi$ ranges over generic characters of $G$.  By GIT, there is a geometric quotient $V^{\chi-s}/\!/G = V^{\chi-ss}/\!/G$.  Since the $G$-orbits on $V^{\chi-s}$ are closed, passing to the $G$-stable open subset $U^{\chi-s}$ yields a  geometric quotient $U^{\chi-s}/\!/G$.  
For varying $\chi$, these can be glued together along intersections to obtain a geometric quotient prevariety $U/\!/G$. By functoriality this canonically maps to the categorical quotient $V/\!/G$.
\end{proof}
We now return to $U$ as before, and show that  $U/\!/G_\alpha$ is smooth. 
Equivalently, we show that $U/\!/PG_\alpha$ is smooth. It suffices to show that the stabilizer in $PG_\alpha$
of every point of $U$ is trivial.  Indeed, for $g = (g_i) \in G_\alpha$ to stabilize a nonzero path from vertex $i$ to $j$ (with $i \neq j$), we must have $g_i=g_j$. So $g$ is in the diagonal $\bC^\times$ (meaning $g_i = g_j$ for all $i, j \in Q_0$) if the set of nonzero arrows in $U$, considered as an undirected graph, is connected.
By Definition \ref{d:good_subset} (3) of $U$, there is a nonzero path from $x$ or $y$ to each $i \in \{ 1, \dots, n\}$. And by definition of $\mathcal{J}$, $I \cap J \neq \emptyset$, so there is some $i' \in I \cap J$ receiving nonzero paths from both $x$ and $y$. Put together, the undirected graph of nonzero arrows must be connected.


\subsubsection{Crepancy}
Next, we show that $\pi_U$ is birational and crepant. Note that by our assumptions on $Q$, 
$\alpha \in \Sigma_0$, i.e., 
there exists a simple representation of $\overline{Q}$ in $\mu^{-1}(0)$.
Explicitly, since all dimensions are one, these are the representations for which there is a nonzero path from every vertex to every other vertex. The locus of simple representations $U^s \subseteq U$
consists of closed, free orbits of $PG_\alpha$
and the maps $U^{\text{s}}/\!/G_{\alpha} \to U/\!/G_\alpha$ and $U^{\text{s}}/\!/G_{\alpha} \to \mu^{-1}(0)/\!/G_\alpha$ are open embeddings. This implies that $\pi_U$ is birational.

Now we show that $\pi_U$ is crepant. It suffices to show that the source is symplectic, since the target is a symplectic singularity and has trivial canonical bundle.

To see this, note that $U = \mathcal{U} \cap \mu^{-1}(0)$ where $\mathcal{U}$ is an open subset of the symplectic representation space, $\Rep(\overline{Q},\alpha) = \bigoplus_{a \in \overline{Q}} \Hom(\bC^{\alpha_{s(a)}}, \bC^{\alpha_{t(a)}})$, defined again by the same condition that there exist nonzero paths with specified endpoints. Then $\mathcal{U}$ is naturally symplectic.

Thus, $U/\!/PG_\alpha$ is a Hamiltonian reduction of a smooth symplectic variety $\mathcal{U}$ by a free Hamiltonian action of a reductive group, $PG_\alpha$. As a result it smooth  symplectic.

\subsubsection{The $A_2$ property and separability}

We next prove  that $U/\!/PG_\alpha$ has the $A_2$ property: any two points lie in some affine open subset. We prove the equivalent (but \emph{a priori} more general) assertion that they lie in a quasiprojective open subset. Note that the $A_2$ property immediately implies that $U/\!/PG_\alpha$ is separated, and hence a variety. As a consequence, $\pi_U$ is separated (by \url{https://stacks.math.columbia.edu/tag/01KV}).

The argument is similar to \cite{ADHL}[Proposition 3.1.3.8], but we give a direct, self-contained proof. The key notion is that of an \emph{orbit cone}:
\begin{defn}
Let a reductive group $G$ act on an affine variety $V$.
    The \emph{orbit cone} $\omega_v$ of a point $v \in V$ is the $\bR_{\geq 0}$-linear span in $\Hom(G,\bC^\times) \otimes_\bZ \bR$ of all characters $\theta: G \to \bC^\times$ for which $v$ is in the $\theta$-semistable locus.
\end{defn}
Note that $\omega_x$ is a union of GIT equivalence classes. Two points $\theta, \theta' \in  \Hom(G,\bC^\times) \otimes_\bZ \bR$ are equivalent if and only if,  for every root $\beta \leq \alpha$, which in our situation is merely the restriction of $\alpha$ to a connected subset of $Q$, we have that $\theta \cdot \beta$ and $\theta' \cdot \beta$ have the same sign, interpreted as an element in the set $\{-,0,+\}$. 

We consider $PG_\alpha$ acting on $U$. Note that $U$ is covered by $PG_\alpha$-stable affine open subsets: $U' \subseteq U$, defined by the condition that for certain ordered pairs of vertices $(a,b)$, there is a nonzero path from $a$ to $b$. Fix some $U'$ and let $u \in U'$. 
We claim that $\omega_u \subseteq \Hom(PG_\alpha, \bC^\times) \otimes_\bZ \bR$ is a top-dimensional cone, i.e., its interior is open.  This is a consequence of the fact used in the proof of Lemma \ref{l:fin-stabilisers}  that every point with finite (here trivial) stabilizer is in the stable locus for some generic stability parameter.
\begin{rem}
We can also give a direct argument that $\omega_u$ has open interior:
$U'$ is defined by the condition that there exist nonzero paths from vertex $x$ to a subset $I \subseteq \{1,\ldots,n\}$ of vertices and from $y$ to a subset $J \subseteq \{1,\ldots,n\}$ of vertices. Moreover, we have $I \cup J = \{1,\ldots,n\}$ and $I \cap J \neq \emptyset$. Then we can consider the stability condition $\theta$ which has $\theta_x = |I|$, $\theta_y = |J|$, and $\theta_i = -\frac{|I|+|J|}{n}$ for all $i \in \{1,\ldots,n\}$. Since $\theta_i < -1$ for all $i \in I$, if $\theta \cdot \beta > 0$ for a root $\beta \leq \alpha$, then $\beta$ is either a root $\beta_{I'}$  or a root $\gamma_{J'}$ for some $|I'| < |I|$ or $|J'| < |J|$, or else $\beta > e_x + e_y$. This implies by the definition of $U'$ that  there must exist a nonzero path from a vertex in the support of $\beta$ to a vertex in the complement of its support. Since this holds for all $\beta \leq \alpha$, 
the representation $u$ is indeed $\theta$-stable. Moreover, the argument works for an open ball around $\theta$, so we find that $\omega_u$ contains an open set and hence an open cone, as desired. 
\end{rem}
 Now, given two points $u_1,u_2 \in U$, we claim that the interiors $\omega_{u_1}^\circ, \omega_{u_1}^\circ$ of the orbit cones $\omega_{u_1}, \omega_{u_2}$ must intersect nontrivially. We use the following basic lemma:
 \begin{lem}
Let $\mathcal{H}$ be a collection of hyperplanes in $\bR^n$, and let $C_1, C_2 \subseteq \bR^n$ be two nonempty open polyhedral cones defined as the intersection of open half-spaces defined by some of the hyperplanes in $\mathcal{H}$. The intersection $C_1 \cap C_2$ is empty if and only if, for some $H \in \mathcal{H}$, $C_1$ and $C_2$ are contained in different half-spaces with boundary $H$.
 \end{lem}
 \begin{proof}
Since the boundaries of $C_1$ and $C_2$ are contained in the union of the hyperplanes in $\mathcal{H}$, the intersection $C_1 \cap C_2$ is empty if and only if they are separated by at least one of the planes $H \in \mathcal{H}$. 
     \end{proof}
     Every pair of points $u_1,u_2 \in U$ lie in $PG_\alpha$-stable open subsets $U_1, U_2$ defined by certain ordered pairs of vertices having nonzero paths.  The orbit cones of $u_1$ and $u_2$ are defined as corresponding intersections of half-spaces defined by the root hyperplanes $\beta^\perp$ for $\beta \leq \alpha$.  By construction, for each hyperplane $\beta^\perp$, only one of the half-spaces is possible, i.e., there is at most one sign $+$ or $-$ for which we can require $\theta \cdot \beta$ has that sign in an orbit cone.  Namely, we can require the sign $-$ for roots $\beta_I$ with $I \in \mathcal{I}$ and for roots $\gamma_J$ with $J \in \mathcal{J}$, and $+$ for the complementary roots $\alpha-\beta_I, \alpha-\gamma_J$.  

     As a result, there exists a stability condition $\theta$ in the intersection $\omega_{u_1}^\circ \cap \omega_{u_2}^\circ$. Since the intersection is open, we can find $\theta$ which is not in any of the root hyperplanes $\beta^\perp$ for $\beta \leq \alpha$. Then $PG_\alpha$ acts freely on the $\theta$-semistable locus, which equals the $\theta$-stable locus. As a result the inclusions $U' := U \cap \mu^{-1}(0)^{\theta-\text{s}}
     \to U$ and $U' \to \mu^{-1}(0)^{\theta-\text{s}}$ induce open inclusions after dividing by $PG_\alpha$.
     Since $\mu^{-1}(0)^{\theta-\text{s}}/\!/PG_\alpha = \cM_\theta(Q,\alpha)$ is quasi-projective, we find that the images of $u_1,u_2$ in $U/\!/PG_\alpha$ are contained in an open subset which is quasi-projective. Hence $U/\!/PG_\alpha$ has the $A_2$ property as desired.

\subsubsection{Properness}
Finally, we come to the main step of the proof: to show that $\pi_U$ is proper. We will apply the valuative criterion of properness together with Hilbert's description of maximal ideals of invariant rings (i.e., Lemma \ref{l:good-quotient}).  We write the valuative criterion as follows:
\begin{lem} Let $V$ be an affine variety, $a: G \times V \to V$ an action of a reductive algebraic group $G$  on $V$, and $U \subseteq V$ a $G$-stable open subset.  Then the morphism $\pi_U: U/\!/G \to V/\!/G$ is proper if it is separated and, for every discrete valuation ring $R$ with fraction field $K$
and every point $p: \Spec K \to U$ lying over a point $q: \Spec R \to V/\!/G$, there exists an element $g \in G(K)$ and a point $\widetilde q: \Spec R \to U$ making the diagram commutative: \[\begin{tikzcd}
    \Spec K \arrow[r,"g \times p"] \arrow[d] & G \times U \arrow[r,"a"] & U  \arrow[d] \\
    \Spec R \arrow[rru,dashed,"\widetilde q"] \arrow[rr,"q"] & & V/\!/G
\end{tikzcd}
\]
\end{lem}
\begin{proof} Since we assume that $\pi_U$
is separated, the usual valuative criterion of properness asks that, for every pair of points $\overline{p}: \Spec K \to U/\!/G$ and $q: \Spec R \to V/\!/G$ fitting into a commutative diagram, there exists a $\hat q: \Spec R \to U/\!/G$ lifting the two, i.e., a dashed arrow completing the following diagram exists:
\[
\begin{tikzcd}
\Spec K \arrow[r,"\overline{p}"] \arrow[d] & U/\!/G \arrow[d,"\pi_U"] \\
\Spec R \arrow[r,"q"] \arrow[ru,"\hat q",dashed] & V/\!/G
\end{tikzcd}
\]

Suppose that we begin with a point $\overline{p}: \Spec K \to U/\!/G$. If there exists a point $p: \Spec K \to U$ whose composition with the quotient morphism is $\overline{p}$, then by the assumption we obtain a point $\widetilde q: \Spec R \to U$ whose composition with the quotient morphism yields a point $\hat q: \Spec R \to U/\!/G$ 
satisfying the usual valuative criterion of properness.

Thus it would suffice to show that the point $p: \Spec K \to U$ exists. However, this point does not have to exist in general. Rather, it is a consequence of Lemma \ref{l:good-quotient}
 that, for some finite extension $K' \supseteq K$, there is a point $p': \Spec K' \to U$ fitting into a commutative diagram:
\[\begin{tikzcd}
  \Spec K' \arrow[r] \arrow[d] & U \arrow[d] \\
  \Spec K \arrow[r] & U/\!/G
\end{tikzcd}
\]
Now we can let $R' \subseteq K'$ be the integral closure of $R$ in $K'$, and by applying the lemma, we find a point $q': \Spec R' \to U$ satisfying the desired hypotheses.  The image $\Spec R' \to U/\!/G$ factors through a $K$ point. Since $R' \cap K =R$, it also factors through an $R$-point $\hat q: \Spec R \to U/\!/G$ verifying the valuative criterion of properness.
\end{proof}

Next, we explain how to verify the hypotheses of the lemma in the situation at hand. Given a $K$-point $u \in U(K)$, we will apply elements of $G_\alpha(K)$ in a systematic way in order to produce an element of $U(R)$.  We apply the following lemma repeatedly:
\begin{lem} \label{l:proper}
Let $Q$ be a quiver
and $I \subseteq Q_0$ a subset of the vertices. Let $R$ be a discrete valuation ring with fraction field $K$.
 Suppose that $p \in \Rep(\overline{Q},\alpha)(K)$ has the properties:
\begin{itemize}
    \item all paths with both endpoints in $I$ are valued in $R$;
    \item 
    there is a nonzero path from some vertex in $I$ to some vertex in $I^c \subseteq Q_0$.
\end{itemize} 
Then there is an element in the diagonal subgroup, $g \in K^\times \subseteq (K^\times)^I$, and a vertex $j \in I^c$ such that the representation $g \cdot p$ satisfies:
\begin{itemize}
    \item all paths beginning at vertices in $I \cup \{j\}$ are valued in $R$;
    \item some path from a vertex in $I$ to $j$ is not valued in the maximal ideal $\mathfrak{m} \subseteq R$.
\end{itemize}
If, in the representation $p$, all paths beginning at $I$ are valued in $R$, then $g^{-1} \in R$, so that this operation preserves the condition that any path is valued in $R$.
\end{lem}
\begin{proof}
Let $\Delta_I: K^\times \to (K^\times)^I$ be the diagonal embedding. For $c \in K^\times$, note that applying $\Delta_I(c)$ to a representation rescales paths from $I$ to $I^c$ by $c^{-1}$, those from $I^c$ to $I$ by $c$, and leaves other paths fixed.

Let $t \in R$ be a uniformizer.  In view of the preceding paragraph, there exists a maximum integer $m \in \bZ$ such that $\Delta_I(t^m) \cdot p$ has 
values of all arrows from $I$ to $I^c$ in $R$. Then there exists an arrow from $I$ to $I^c$ which is not in the maximal ideal. We can let $j$ be the target of this arrow.  

We claim that $j$ has the desired properties.  By construction, it satisfies all properties except possibly the condition that all paths beginning at $j$ are valued in $R$. Suppose some such path is not valued in $R$. Then it is in $t^{-m} R$ for some $m \geq 1$.  Composing with a path beginning at a vertex of $I$ and ending at $j$, which is valued in $R$ but not in $tR$, we get a path beginning at $I$ and valued in $t^{-m} R$, a contradiction.

The last statement holds by our choice of $m$, and because we already established that all paths beginning at $I$ will be valued in $R$. \qedhere
\end{proof} 

Finally, we return to our original quiver with $\alpha=(1,\ldots,1)$. Begin with a representation $p \in \Rep(\overline{Q},\alpha)(K)$ whose image in $\Rep(\overline{Q},\alpha)/\!/G_\alpha$ is an $R$-point. The latter condition means that (the trace of) every closed path is in $R$. 

Let $I := \{x\}$ so every path beginning and ending at $x$ lies in $R$. 
By Lemma \ref{l:proper}, passing to other representations in the $G_{\alpha}(K)$ orbit of $p$, we can continually expand $I$ until either $I$ contains all vertices, or there are no nonzero paths from $I$ to $I^c$. All paths beginning at $I$ will be valued in $R$, and there will be paths from $x$ to all vertices of $I$ which are not in the maximal ideal (from $x$ to itself, this can be the identity path).  

Next, we set $J := \{y\}$ and apply the same procedure. The first step could create some paths ending at $y$ which are not in $R$, but after this step, by the last statement of the lemma, we will not create any new paths which are not in $R$. Note that every time we replace $J$ with $J \cup \{j\}$ for $j \in \{1,\ldots,n\}$, we have to remove the vertices of $J$ (other than $j$) from $I$. If at some point we add $x$ to $J$, then by Definition \ref{d:good_subset} (3) of $U$, we can continue the procedure until $J$ contains all vertices. If we do not add $x$ to $J$, we can continue the procedure until $I \cup J$ consists of all vertices. 

The end result of the preceding paragraph is that there are subsets $I':=I\setminus\{x\}, J':=J\setminus\{y\} \subseteq \{1,\ldots,n\}$, with $J'$ nonempty, such that:
\begin{itemize} 
\item $I' \cup J' = \{1,\ldots,n\}$;
\item All paths not ending at $y$
are in $R$;
\item For each $i \in I'$ there is a path from $x$ to $i$ which is not in $\mathfrak{m} \subseteq R$;
\item For each $j \in J'$ there is a path from $y$ to $j$ which is not in $\mathfrak{m} \subseteq R$;
\item $J' \in \mathcal{J}$.
\end{itemize}
Suppose there is a path ending at $y$ which is \emph{not} in $R$. Then composing with a path from $y$ to a vertex of $J'$ in $R \setminus \mathfrak{m}$ would produce a path not ending at $y$ which is not valued in $R$, a contradiction. So all paths are valued in $R$.


If $I' \in \mathcal{I}$, then by Definition \ref{d:good_subset}, the representation is in $U(R)$, which completes the proof that $\pi_U$ is proper. 

Otherwise, we again apply the Lemma \ref{l:proper} to $I:= \{x\} \cup I'$, at each stage either adding (a) the vertex $y$, or (b) a vertex $i$ of $\{1,\ldots,n\}$. 
\begin{itemize}
    \item[-] If at any stage we add $y$ to the set $I$ then we can freely add the entire set $J'$ to $I'$. So $I' = \{ 1, \ldots, n \}$. Since $\mathcal{I}$ is nonempty and upward-closed , $I' = \{1, \ldots, n \} \in \mathcal{I}$, completing the proof. 
    \item[-] If we wish to add $i \in \{1 , \ldots, n \}$ to $I$, then we first need to remove all vertices of $I'$ from $J'$; in other words, we set $J':=\{1,\ldots,n\}\setminus I'$ and afterwards add $i$ to $I'$. After this we still have $J' \in \mathcal{J}$, by the defining property of $\mathcal{J}$, because $I' \not \in \mathcal{I}$. We iterate until $I' \in \mathcal{I}$, again completing the proof.
\end{itemize}


\bibliographystyle{hyperamsalpha}
\bibliography{ProperResolutions}

\providecommand{\bysame}{\leavevmode\hbox to3em{\hrulefill}\thinspace}
\providecommand{\MR}{\relax\ifhmode\unskip\space\fi MR }
\providecommand{\MRhref}[2]{%
  \href{http://www.ams.org/mathscinet-getitem?mr=#1}{#2}
}
\providecommand{\href}[2]{#2}
\begin{thebibliography}{BCHM10}

\bibitem[ADHL15]{ADHL}
I.~Arzhantsev, U.~Derenthal, J.~Hausen, and A.~Laface, \emph{Cox rings},
  Cambridge Studies in Advanced Mathematics, vol. 144, Cambridge University
  Press, Cambridge, 2015.

\bibitem[AP16]{AP16}
M.~Arbo and N.~Proudfoot, \emph{Hypertoric varieties and zonotopal tilings},
  \href{http://dx.doi.org/10.1093/imrn/rnw005}{Int. Math. Res. Not. IMRN
  (2016)}, no.~23, 7268--7301.

\bibitem[BCHM10]{BCHM}
C.~Birkar, P.~Cascini, C.~D. Hacon, and J.~McKernan, \emph{Existence of minimal
  models for varieties of log general type},
  \href{http://dx.doi.org/10.1090/S0894-0347-09-00649-3}{J. Amer. Math. Soc.
  \textbf{23} (2010)}, no.~2, 405--468.

\bibitem[BCS23]{BCS23}
G.~Bellamy, A.~Craw, and T.~Schedler, \emph{Birational geometry of quiver
  varieties and other {G}{I}{T} quotients}, 2023.
  \href{http://arxiv.org/abs/2212.09623}{{\tt arXiv:2212.09623 [math.AG]}}.

\bibitem[Bea00]{Beauville}
A.~Beauville, \emph{Symplectic singularities},
  \href{http://dx.doi.org/10.1007/s002229900043}{Invent. Math. \textbf{139}
  (2000)}, no.~3, 541--549.

\bibitem[BS21]{BS21}
G.~Bellamy and T.~Schedler, \emph{Symplectic resolutions of quiver varieties},
  \href{http://dx.doi.org/10.1007/s00029-021-00647-0}{Selecta Math. (N.S.)
  \textbf{27} (2021)}, no.~3, Paper No. 36, 50.

\bibitem[CB01]{CB01}
W.~Crawley-Boevey, \emph{Geometry of the moment map for representations of
  quivers}, \href{http://dx.doi.org/10.1023/A:1017558904030}{Compositio Math.
  \textbf{126} (2001)}, no.~3, 257--293.

\bibitem[CB03]{CB03}
\bysame, \emph{Normality of {M}arsden-{W}einstein reductions for
  representations of quivers},
  \href{http://dx.doi.org/10.1007/s00208-002-0367-8}{Math. Ann. \textbf{325}
  (2003)}, no.~1, 55--79.

\bibitem[Dr{\'e}04]{Drezet-Luna-slice}
J.-M. Dr{\'e}zet, \emph{Luna's slice theorem and applications}, Algebraic group
  actions and quotients, Hindawi Publ. Corp., Cairo, 2004, pp.~39--89.

\bibitem[Fu03]{Fu}
B.~Fu, \emph{Symplectic resolutions for nilpotent orbits},
  \href{http://dx.doi.org/10.1007/s00222-002-0260-9}{Invent. Math. \textbf{151}
  (2003)}, no.~1, 167--186.

\bibitem[Hub24]{Hubbard}
A.~Hubbard, \emph{All crepant resolutions of hyperpolygon spaces via their
  {C}ox rings}, 2024. \href{http://arxiv.org/abs/2406.04117}{{\tt
  arXiv:2406.04117 [math.AG]}}.

\bibitem[Kal03]{Kaledin_Crepant}
D.~Kaledin, \emph{On crepant resolutions of symplectic quotient singularities},
  \href{http://dx.doi.org/10.1007/s00029-003-0308-8}{Selecta Math. (N.S.)
  \textbf{9} (2003)}, no.~4, 529--555.

\bibitem[Kal06]{Kaledin_Poisson}
\bysame, \emph{Symplectic singularities from the {P}oisson point of view},
  \href{http://dx.doi.org/10.1515/CRELLE.2006.089}{J. Reine Angew. Math.
  \textbf{600} (2006)}, 135--156.

\bibitem[Kal09]{Kaledin-survey}
\bysame, \href{http://dx.doi.org/10.1090/pspum/080.2/2483948}{\emph{Geometry
  and topology of symplectic resolutions}}, Algebraic geometry---{S}eattle
  2005. {P}art 2, Proc. Sympos. Pure Math., vol. 80, Part 2, Amer. Math. Soc.,
  Providence, RI, 2009, pp.~595--628.

\bibitem[Kam22]{Kamnitzer-survey}
J.~Kamnitzer, \emph{Symplectic resolutions, symplectic duality, and {C}oulomb
  branches}, \href{http://dx.doi.org/10.1112/blms.12711}{Bull. Lond. Math. Soc.
  \textbf{54} (2022)}, no.~5, 1515--1551.

\bibitem[Kin94]{King}
A.~D. King, \emph{Moduli of representations of finite-dimensional algebras},
  \href{http://dx.doi.org/10.1093/qmath/45.4.515}{Quart. J. Math. Oxford Ser.
  (2) \textbf{45} (1994)}, no.~180, 515--530.

\bibitem[KS24]{KS24}
D.~Kaplan and T.~Schedler, \emph{Crepant resolutions of stratified varieties
  via gluing}, \href{http://dx.doi.org/10.1093/imrn/rnae135}{Int. Math. Res.
  Not. IMRN \textbf{2024} (2024)}, no.~17, 12161--12200.

\bibitem[MFK94]{GIT}
D.~Mumford, J.~Fogarty, and F.~Kirwan, \emph{Geometric invariant theory}, 3rd
  ed., Ergebnisse der Mathematik und ihrer Grenzgebiete (2), vol.~34,
  Springer-Verlag, Berlin, 1994.

\bibitem[MN18]{MN_Kirwansurj}
K.~McGerty and T.~Nevins, \emph{Kirwan surjectivity for quiver varieties},
  \href{http://dx.doi.org/10.1007/s00222-017-0765-x}{Invent. Math. \textbf{212}
  (2018)}, no.~1, 161--187.

\bibitem[Mum88]{Mumford}
D.~Mumford, \href{http://dx.doi.org/10.1007/978-3-662-21581-4}{\emph{The red
  book of varieties and schemes}}, Lecture Notes in Mathematics, vol. 1358,
  Springer-Verlag, Berlin, 1988.

\bibitem[Nak94]{Nakajima}
H.~Nakajima, \emph{Instantons on {ALE} spaces, quiver varieties, and
  {K}ac-{M}oody algebras},
  \href{http://dx.doi.org/10.1215/S0012-7094-94-07613-8}{Duke Math. J.
  \textbf{76} (1994)}, no.~2, 365--416.

\bibitem[Nak98]{Nakajima98}
H.~Nakajima, \emph{{Quiver varieties and Kac-Moody algebras}},
  \href{http://dx.doi.org/10.1215/S0012-7094-98-09120-7}{Duke Math. J.
  \textbf{91} (1998)}, no.~3, 515 -- 560.

\bibitem[Nam15]{Namikawa}
Y.~Namikawa, \emph{Poisson deformations and birational geometry}, J. Math. Sci.
  Univ. Tokyo \textbf{22} (2015), no.~1, 339--359.

\bibitem[New78]{Newstead}
P.~E. Newstead, \emph{Introduction to moduli problems and orbit spaces},
  vol.~51, Tata Institute of Fundamental Research, Bombay; Narosa Publishing
  House, New Delhi, 1978.

\bibitem[Oko18]{Okounkov-survey}
A.~Okounkov, \href{http://dx.doi.org/10.1142/9789813272880_0030}{\emph{On the
  crossroads of enumerative geometry and geometric representation theory}},
  Proceedings of the {I}nternational {C}ongress of {M}athematicians (ICM 2018),
  World Sci. Publ., Hackensack, NJ, 2018, pp.~839--867.

\bibitem[Pro11]{Proudfoot-GIT-at-once}
N.~Proudfoot, \emph{All the {GIT} quotients at once},
  \href{http://dx.doi.org/10.1090/S0002-9947-2010-05483-2}{Trans. Amer. Math.
  Soc. \textbf{363} (2011)}, no.~4, 1687--1698.

\end{thebibliography}

\end{document}